\documentclass[reqno]{amsart}
\usepackage{amsthm, amssymb, amsfonts}
\usepackage{lmodern}
\usepackage{color}
\usepackage{hyperref}
\usepackage[T5]{fontenc}
\usepackage{amscd,amssymb}
\usepackage[v2,cmtip]{xy}
\usepackage{mathrsfs}
\theoremstyle{plain}
\newtheorem{thm}{Theorem}[section]
\newtheorem{prop}[thm]{Proposition}
\newtheorem{lem}[thm]{Lemma}

\theoremstyle{definition}
\newtheorem{defn}[thm]{Definition}

\newtheorem{nota}[thm]{Notation}

\theoremstyle{plain}
\newtheorem{thms}{Theorem}[subsection]
\newtheorem{props}[thms]{Proposition}
\newtheorem{lems}[thms]{Lemma}

\theoremstyle{definition}

\begin{document} 
\title[On a minimal set of generators for the polynomial algebra]
{On a minimal set of generators for\\ the polynomial algebra of five variables\\ as a module over the Steenrod algebra}

 \author{\DJ\d{\u a}ng V\~o Ph\'uc and Nguy\~\ecircumflex n Sum$^{1}$}

\footnotetext[1]{Corresponding author.}
\footnotetext[2]{2000 {\it Mathematics Subject Classification}. Primary 55S10; 55S05, 55T15.}
\footnotetext[3]{{\it Keywords and phrases:} Steenrod squares, Peterson hit problem, Polynomial algebra.}

\bigskip
\begin{abstract}
 Let $P_k$ be the graded polynomial algebra $\mathbb F_2[x_1,x_2,\ldots ,x_k]$ over the prime field of two elements, $\mathbb F_2$, with the degree of each $x_i$ being 1. We study the {\it hit problem}, set up by Frank Peterson, of determining a minimal set of generators for $P_k$ as a module over the mod-$2$ Steenrod algebra, $\mathcal{A}.$ In this paper, we explicitly determine a minimal set of $\mathcal{A}$-generators for $P_k$ in the case $k=5$ and the degree $4(2^d - 1)$ with $d$ an arbitrary positive integer.
\end{abstract}
\maketitle

\section{Introduction}\label{s1}
\setcounter{equation}{0}

 Let $P_k$ be the graded polynomial algebra $\mathbb F_2[x_1,x_2,\ldots ,x_k]$, with the degree of each $x_i$
being 1. This algebra arises as the cohomology with coefficients in $\mathbb F_2$ of an elementary abelian 2-group of rank $k$. Then, $P_k$ is a module over the mod-2 Steenrod algebra, $\mathcal A$.  
The action of $\mathcal A$ on $P_k$ is determined by the elementary properties of the Steenrod squares $Sq^i$ and subject to the Cartan formula (see Steenrod and Epstein~\cite{st}).

An element $g$ in $P_k$ is called {\it hit} if it belongs to  $\mathcal{A}^+P_k$, where $\mathcal{A}^+$ is the augmentation ideal of $\mathcal A$. That means $g$ can be written as a finite sum $g = \sum_{u\geqslant 0}Sq^{2^u}(g_u)$ for suitable polynomials $g_u \in P_k$.  

We study the {\it Peterson hit problem} of determining a minimal set of generators for the polynomial algebra $P_k$ as a module over the Steenrod algebra. In other words, we want to determine a basis of the $\mathbb F_2$-vector space $QP_k := P_k/\mathcal A^+P_k = \mathbb F_2 \otimes_{\mathcal A} P_k$. The problem  is an  interesting  and  important  one.
It was first studied by Peterson~\cite{pe}, Wood~\cite{wo}, Singer~\cite {si1}, 
and Priddy~\cite{pr}, who showed its relation to several classical problems respectively in cobordism theory, modular representation theory, the Adams spectral sequence for the stable homotopy of spheres, and stable homotopy type of classifying spaces of finite groups. Then, this problem  was investigated by many authors (see Boardman \cite{bo}, Bruner, H\`a and H\uhorn ng ~\cite{br}, Carlisle and Wood~\cite{cw}, Crabb and Hubbuck~\cite{ch}, H\uhorn ng \cite{hu}, H\uhorn ng and Nam~\cite{hn2}, Janfada and Wood~\cite{jw1}, Kameko~\cite{ka}, Mothebe \cite{mo}, Nam~\cite{na}, Repka and Selick~\cite{res}, Silverman~\cite{sl}, Silverman and Singer~\cite{ss}, Singer~\cite{si2}, Walker and Wood~\cite{wa,wa2}, Wood~\cite{wo2}, the second named author \cite{su2,su5} and others).

From the results of Wood \cite{wo} and Kameko \cite{ka}, the hit problem is reduced to the case of degree $n$ of the form
\begin{equation} \label{ct1.1}n =  s(2^d-1) + 2^dm,
\end{equation}
where $s, d, m$ are non-negative integers and $1 \leqslant s <k$ (see \cite{su5}.) For $s=k-1$ and $m > 0$, the problem was studied by Crabb and Hubbuck~\cite{ch}, Nam~\cite{na}, Repka and Selick~\cite{res} and the second named author ~\cite{su2,su5}.  

In the present paper, we explicitly determine the hit problem in degree $n$ of the form (\ref{ct1.1}) with $s=k-1=4$, $m=0$ and $d$ an arbitrary positive integer. The main result of the paper is the following.

\medskip\noindent
{\bf Main Theorem.} {\it  Let $n= 4(2^d-1)$ with $d$ a positive integer. The dimension of the $\mathbb F_2$-vector space $(QP_5)_{n}$ is determined by the following table}:

\medskip
\centerline{\begin{tabular}{c|ccccc}
$n = 4(2^d-1)$&$d=1$ & $d=2$ & $d=3$ & $d=4$ & $d\geqslant 5$\cr
\hline
\ $\dim(QP_5)_n$ & $45$ & $190$ & $480$ &$650$& $651$ \cr
\end{tabular}}

\bigskip
 In Section \ref{s2}, we recall  some needed information on the admissible monomials in $P_k$, Singer's criterion on the hit monomials  and Kameko's homomorphism.  The proof of Main Theorem is presented in Section~ \ref{s3}.

\section{Preliminaries}\label{s2}
\setcounter{equation}{0}

In this section, we recall some needed information from Kameko~\cite{ka}, Singer \cite{si2} and the second named author \cite{su5} which will be used in the next section.

\begin{nota} We denote $\mathbb N_k = \{1,2, \ldots , k\}$ and
\begin{align*}
X_{\mathbb J} = X_{\{j_1,j_2,\ldots , j_s\}} =
 \prod_{j\in \mathbb N_k\setminus \mathbb J}x_j , \ \ \mathbb J = \{j_1,j_2,\ldots , j_s\}\subset \mathbb N_k,
\end{align*}
In particular, we set
\begin{align*}
&X_{\mathbb N_k} =1,\\
&X_\emptyset = x_1x_2\ldots x_k,\\ 
&X_j = x_1\ldots \hat x_j \ldots x_k, \ 1 \leqslant j \leqslant k,
\end{align*}
and  $X:=X_k \in P_{k-1}.$

Let $\alpha_i(a)$ denote the $i$-th coefficient  in dyadic expansion of a non-negative integer $a$. That means
$$a= \alpha_0(a)2^0+\alpha_1(a)2^1+\alpha_2(a)2^2+ \ldots ,$$ 
for $ \alpha_i(a) =0$ or 1 with $i\geqslant 0$. 

Let $x=x_1^{a_1}x_2^{a_2}\ldots x_k^{a_k} \in P_k$. Denote $\nu_j(x) = a_j, 1 \leqslant j \leqslant k$.  
Set 
$$\mathbb J_t(x) = \{j \in \mathbb N_k :\alpha_t(\nu_j(x)) =0\},$$
for $t\geqslant 0$. Then, we have
$x = \prod_{t\geqslant 0}X_{\mathbb J_t(x)}^{2^t}.$ 
\end{nota}
\begin{defn}
For a monomial  $x \in P_k$,  define two sequences associated with $x$ by
\begin{align*} 
\omega(x)&=(\omega_1(x),\omega_2(x),\ldots , \omega_i(x), \ldots),\\
\sigma(x) &= (\nu_1(x),\nu_2(x),\ldots ,\nu_k(x)),
\end{align*}
where
$\omega_i(x) = \sum_{1\leqslant j \leqslant k} \alpha_{i-1}(\nu_j(x))= \deg X_{\mathbb J_{i-1}(x)},\ i \geqslant 1.$
The sequences $\omega(x)$ and $\sigma(x)$ are respectively called  the weight vector and the exponent vector of $x$. 

Let $\omega=(\omega_1,\omega_2,\ldots , \omega_i, \ldots)$ be a sequence of non-negative integers.  The sequence $\omega$  is called  the weight vector if $\omega_i = 0$ for $i \gg 0$.
\end{defn}

The sets of the weight vectors and the exponent vectors are given the left lexicographical order. 

  For a  weight vector $\omega$,  we define $\deg \omega = \sum_{i > 0}2^{i-1}\omega_i$.  If  there are $i_0=0, i_1, i_2, \ldots , i_r > 0$ such that 
\begin{align*}
&i_1 + i_2 + \ldots + i_r  = m,\\
&\omega_{i_0+\ldots +i_{s-1} + t} = b_s, 1 \leqslant t \leqslant i_s, 1 \leqslant s \leqslant  r, 
\end{align*}
and $\omega_i=0$ for all $i > m$, then we write $\omega = (b_1^{(i_1)},b _2^{(i_2)},\ldots , b_r^{(i_r)})$. Denote $b_u^{(1)} = b_u$.  For example, $\omega = (3,3,2,2,2,1,0,\ldots) = (3^{(2)},2^{(3)},1)$.

Denote by   $P_k(\omega)$ the subspace of $P_k$ spanned by all monomials $y$ such that
$\deg y = \deg \omega$, $\omega(y) \leqslant \omega$, and by $P_k^-(\omega)$ the subspace of $P_k$ spanned by all monomials $y \in P_k(\omega)$  such that $\omega(y) < \omega$. 

\begin{defn}\label{dfn2} Let $\omega$ be a weight vector and $f, g$ two polynomials  of the same degree in $P_k$. 

i) $f \equiv g$ if and only if $f - g \in \mathcal A^+P_k$. If $f \equiv 0$, then $f$ is called {\it hit}.

ii) $f \equiv_{\omega} g$ if and only if $f - g \in \mathcal A^+P_k+P_k^-(\omega)$. 
\end{defn}

Obviously, the relations $\equiv$ and $\equiv_{\omega}$ are equivalence ones. Denote by $QP_k(\omega)$ the quotient of $P_k(\omega)$ by the equivalence relation $\equiv_\omega$. Then, we have 
$$QP_k(\omega)= P_k(\omega)/ ((\mathcal A^+P_k\cap P_k(\omega))+P_k^-(\omega)).$$  

For a  polynomial $f \in  P_k$, we denote by $[f]$ the class in $QP_k$ represented by $f$. If  $\omega$ is a weight vector and $f \in  P_k(\omega)$, then denote by $[f]_\omega$ the class in $QP_k(\omega)$ represented by $f$. Denote by $|S|$ the cardinal of a set $S$.

It is easy to see that
$$ QP_k(\omega) \cong QP_k^\omega := \langle\{[x] \in QP_k : x \text{ \rm  is admissible and } \omega(x) = \omega\}\rangle. $$
So, we get
$$(QP_k)_n = \bigoplus_{\deg \omega = n}QP_k^\omega \cong \bigoplus_{\deg \omega = n}QP_k(\omega).$$
Hence, we can identify the vector space $QP_k(\omega)$ with $QP_k^\omega \subset QP_k$. 

\begin{defn}\label{defn3} 
Let $x, y$ be monomials of the same degree in $P_k$. We say that $x <y$ if and only if one of the following holds:  

i) $\omega (x) < \omega(y)$;

ii) $\omega (x) = \omega(y)$ and $\sigma(x) < \sigma(y).$
\end{defn}

\begin{defn}
A monomial $x$ in $P_k$ is said to be inadmissible if there exist monomials $y_1,y_2,\ldots, y_t$ such that $y_j<x$ for $j=1,2,\ldots , t$ and $x - \sum_{j=1}^ty_j \in \mathcal A^+P_k.$ 
A monomial $x$ is said to be admissible if it is not inadmissible.
\end{defn} 

Obviously, the set of all the admissible monomials of degree $n$ in $P_k$ is a minimal set of $\mathcal{A}$-generators for $P_k$ in degree $n$. 
\begin{defn} 
 A monomial $x$ in $P_k$ is said to be strictly inadmissible if and only if there exist monomials $y_1,y_2,\ldots, y_t$ such that $y_j<x,$ for $j=1,2,\ldots , t$ and 
$$x = \sum_{j=1}^t y_j + \sum_{u=1}^{2^s-1}Sq^u(h_u)$$ 
with $s = \max\{i : \omega_i(x) > 0\}$ and suitable polynomials $h_u \in P_k$.
\end{defn}

It is easy to see that if $x$ is strictly inadmissible, then it is inadmissible.

\begin{thm}\label{dlcb1} {\rm (Kameko \cite{ka}, Sum \cite{su2})}  
 Let $x, y, w$ be monomials in $P_k$ such that $\omega_i(x) = 0$ for $i > r>0$, $\omega_s(w) \ne 0$ and   $\omega_i(w) = 0$ for $i > s>0$.

{\rm i)}  If  $w$ is inadmissible, then  $xw^{2^r}$ is also inadmissible.

{\rm ii)}  If $w$ is strictly inadmissible, then $wy^{2^{s}}$ is also strictly inadmissible.
\end{thm} 

Now, we recall a result of Singer \cite{si2} on the hit monomials in $P_k$. 

\begin{defn}\label{spi}  A monomial $z$ in $P_k$   is called a spike if $\nu_j(z)=2^{d_j}-1$ for $d_j$ a non-negative integer and $j=1,2, \ldots , k$. If $z$ is a spike with $d_1>d_2>\ldots >d_{r-1}\geqslant d_r>0$ and $d_j=0$ for $j>r,$ then it is called the minimal spike.
\end{defn}

For a positive integer $n$, by $\mu(n)$ one means the smallest number $r$ for which it is possible to write $n = \sum_{1\leqslant i\leqslant r}(2^{d_i}-1),$ where $d_i >0$. 
In \cite{si2}, Singer showed that if $\mu(n) \leqslant k$, then there exists uniquely a minimal spike of degree $n$ in $P_k$. 
The following is a criterion for the hit monomials in $P_k$.

\begin{thm}\label{dlsig} {\rm (Singer~\cite{si2})}
Suppose $x \in P_k$ is a monomial of degree $n$, where $\mu(n) \leqslant k$. Let $z$ be the minimal spike of degree $n$. If $\omega(x) < \omega(z)$, then $x$ is hit.
\end{thm}

This result implies the one of Wood, which originally is a conjecture of Peterson~\cite{pe}.
 
\begin{thm}\label{dlmd1} {\rm (Wood~\cite{wo})}
If $\mu(n) > k$, then $(QP_k)_n = 0$.
\end{thm} 

One of the main tools in the study of the hit problem is  Kameko's homomorphism 
$\widetilde{Sq}^0_*: QP_k \to QP_k$. 
This homomorphism is induced by the $\mathbb F_2$-linear map, also denoted by  $\widetilde{Sq}^0_*:P_k\to P_k$, given by
$$
\widetilde{Sq}^0_*(x) = 
\begin{cases}y, &\text{if }x=x_1x_2\ldots x_ky^2,\\  
0, & \text{otherwise,} \end{cases}
$$
for any monomial $x \in P_k$. Note that $\widetilde{Sq}^0_*$ is not an $\mathcal A$-homomorphism. However, 
$\widetilde{Sq}^0_*Sq^{2t} = Sq^{t}\widetilde{Sq}^0_*,$ and $\widetilde{Sq}^0_*Sq^{2t+1} = 0$ for any non-negative integer $t$. 

Denote  by
$(\widetilde{Sq}^0_*)_{(k,m)}:(QP_k)_{2m+k}\longrightarrow (QP_k)_m$
Kameko's homomorphism in degree $2m+k$.

\begin{thm}\label{dlmd2} {\rm(Kameko~\cite{ka})}
Let $m$ be a positive integer. If $\mu(2m+k)=k$, then 
$$(\widetilde{Sq}^0_*)_{(k,m)}: (QP_k)_{2m+k}\longrightarrow (QP_k)_m$$
is an isomorphism of the $\mathbb F_2$-vector spaces. 
\end{thm}

By combining Theorems \ref{dlmd1} and \ref{dlmd2}, one can see that the hit problem is reduced to the case of degree $n$ of the form (\ref{ct1.1}) as given in the introduction.

We set 
\begin{align*} 
P_k^0 &=\langle\{x=x_1^{a_1}x_2^{a_2}\ldots x_k^{a_k} \ : \ a_1a_2\ldots a_k=0\}\rangle,
\\ P_k^+ &= \langle\{x=x_1^{a_1}x_2^{a_2}\ldots x_k^{a_k} \ : \ a_1a_2\ldots a_k>0\}\rangle. 
\end{align*}

It is easy to see that $P_k^0$ and $P_k^+$ are the $\mathcal{A}$-submodules of $P_k$. Furthermore, we have the following.

\begin{prop}\label{2.7} We have a direct summand decomposition of the $\mathbb F_2$-vector spaces
$$QP_k =QP_k^0 \oplus  QP_k^+.$$
Here $QP_k^0 = \mathbb F_2\otimes_{\mathcal A}P_k^0$ and  $QP_k^+ = \mathbb F_2\otimes_{\mathcal A}P_k^+$.
\end{prop}

\begin{nota} From now on, we denote by $B_{k}(n)$ the set of all admissible monomials of degree $n$  in $P_k$, 
$$B_{k}^0(n) = B_{k}(n)\cap P_k^0,\ B_{k}^+(n) = B_{k}(n)\cap P_k^+.$$ 
For a weight vector $\omega$ of degree $n$, we set $$B_k(\omega) = B_{k}(n)\cap P_k(\omega),\ B_k^+(\omega) = B_{k}^+(n)\cap P_k(\omega).$$
For a subset $S \subset P_k$, we denote $[S] = \{[f] : f \in S\}$. If $S \subset P_k(\omega)$, then we set 
$$[S]_\omega = \{[f]_\omega : f \in S\}.$$ 
Then, $[B_k(\omega)]_\omega$ and $[B_k^+(\omega)]_\omega$, are respectively the basses of the $\mathbb F_2$-vector spaces $QP_k(\omega)$ and $QP_k^+(\omega) := QP_k(\omega)\cap QP_k^+$.
\end{nota}

Now, we recall some notations and definitions in \cite{su5}.

Denote
$$\mathcal N_k =\big\{(i;I) ; I=(i_1,i_2,\ldots,i_r),1 \leqslant  i < i_1 <  \ldots < i_r\leqslant  k,\ 0\leqslant r <k\big\}.$$

\begin{defn} Let $(i;I) \in \mathcal N_k$, let $r = \ell(I)$ be the length of $I$, and let $u$
be an integer with $1 \leqslant  u \leqslant r$. A monomial $x$ in $P_{k-1}$ is said to be $u$-compatible with $(i;I)$ if all of the following hold:

\smallskip
i) $\nu_{i_1-1}(x)= \nu_{i_2-1}(x)= \ldots = \nu_{i_{(u-1)}-1}(x)=2^{r} - 1$,

ii) $\nu_{i_u-1}(x) > 2^{r} - 1$,

iii) $\alpha_{r-t}(\nu_{i_u-1}(x)) = 1,\ \forall t,\ 1 \leqslant t \leqslant  u$,

iv) $\alpha_{r-t}(\nu_{i_t-1} (x)) = 1,\ \forall t,\ u < t \leqslant r$.
\end{defn}

Clearly, a monomial $x$ can be $u$-compatible with a given $(i;I) \in \mathcal N_k $ for at most one value of $u$. By convention, $x$ is $1$-compatible with $(i;\emptyset)$.

For $1 \leqslant i \leqslant k$, define the homomorphism $f_i: P_{k-1} \to P_k$ of algebras by substituting
$$f_i(x_j) = \begin{cases} x_j, &\text{ if } 1 \leqslant j <i,\\
x_{j+1}, &\text{ if } i \leqslant j <k.
\end{cases}$$

\begin{defn}\label{dfn1} Let $(i;I) \in \mathcal N_k$, $x_{(I,u)} = x_{i_u}^{2^{r-1}+\ldots + 2^{r-u}}\prod_{u< t \leqslant r}x_{i_t}^{2^{r-t}}$ for $r = \ell(I)>0$, $x_{(\emptyset,1)} = 1$.  For a monomial $x$ in $P_{k-1}$, 
we define the monomial  $\phi_{(i;I)}(x)$ in $P_k$ by setting
$$ \phi_{(i;I)}(x) = \begin{cases} (x_i^{2^r-1}f_i(x))/x_{(I,u)}, &\text{if there exists $u$ such that}\\ &\text{$x$ is $u$-compatible with $(i, I)$,}\\
0, &\text{otherwise.}
\end{cases}$$

Then, we have an $\mathbb F_2$-linear map $\phi_{(i;I)}:P_{k-1}\to P_k$.
In particular, $\phi_{(i;\emptyset)} = f_i$.
\end{defn}
For a subset $B \subset P_{k-1}$, denote
\begin{align*}\Phi^0(B) &= \bigcup_{1\leqslant i \leqslant k}\phi_{(i;\emptyset)}(B) = \bigcup_{1\leqslant i \leqslant k}f_i(B).\\
\Phi^+(B) &= \bigcup_{(i;I)\in \mathcal N_k, 0<\ell(I)\leqslant k-1}\phi_{(i;I)}(B)\setminus P_k^0.\\
\Phi(B) &= \Phi^0(B)\bigcup \Phi^+(B). 
\end{align*} 

Clearly, we have
\begin{prop}\label{mdcnd}
If $B$ is a minimal set of generators for  $\mathcal A$-module $P_{k-1}$ in degree $n$, then $\Phi^0(B)$ is also a minimal set of generators for  $\mathcal A$-module $P_k^0$ in degree $n$.
\end{prop}

\begin{thm}\label{mdc1} {\rm(See \cite{su5})} Let $n =\sum_{1 \leqslant i \leqslant k-1}(2^{d_i}-1)$ with $d_i$ positive integers such that $d_1 > d_2 > \ldots >d_{k-2} \geqslant d_{k-1} \geqslant k-1 \geqslant 3$. If $B$ is a minimal set of generators for  $\mathcal A$-module $P_{k-1}$ in degree $n$, then $\Phi(B)$ is also a minimal set of generators for  $\mathcal A$-module $P_k$ in degree $n$. 
\end{thm}

\begin{defn}
For any $(i;I) \in \mathcal N_k$, we define the homomorphism $p_{(i;I)}: P_k \to P_{k-1}$ of algebras by substituting
$$p_{(i;I)}(x_j) =\begin{cases} x_j, &\text{ if } 1 \leqslant j < i,\\
\sum_{s\in I}x_{s-1}, &\text{ if }  j = i,\\  
x_{j-1},&\text{ if } i< j \leqslant k.
\end{cases}$$
Then, $p_{(i;I)}$ is a homomorphism of $\mathcal A$-modules.  In particular, for $I =\emptyset$,  $p_{(i;\emptyset)}(x_i)= 0$  and $p_{(i;I)}(f_i(y)) = y$ for any $y \in P_{k-1}$. 
\end{defn}

\begin{lem}\label{bdm1} {\rm(See \cite{sp})} 
If $x$ is a monomial in $P_k$, then $p_{(i;I)}(x) \in P_{k-1}(\omega(x))$. 
\end{lem}

Lemma \ref{bdm1} implies that if $\omega$ is a weight vector and $x \in P_k(\omega)$, then $p_{(i;I)}(x) \in P_{k-1}(\omega)$. Moreover, $p_{(i;I)}$ passes to a homomorphism from $QP_k(\omega)$ to $QP_{k-1}(\omega)$. 

For a positive integer $b$, denote 
$$\omega_{(k,b)} =((k-1)^{(b)}) , \  \bar \omega_{(k,b)}= ((k-1)^{(b-1)},k-3,1).$$

\begin{prop}\label{mdcm1} {\rm(See \cite{sp})} 
Let $d$ be a positive integer and let $p=\min\{k,d\}$.  Then, the set 
$$B(d) :=\big\{\big[\phi_{(i;I)}(X^{2^d-1})\big]_{\omega_{(k,d)}} : (i;I) \in \mathcal N_{k,p}\big\}$$ 
is a basis of the $\mathbb F_2$-vector space $QP_k(\omega_{(k,d)})$. Consequently 
$$\dim QP_k(\omega_{(k,d)}) = \sum_{t=1}^p\binom kt.$$
\end{prop}

\section{Proof of Main Theorem}\label{s3}

 According to a result in \cite{su5}, $B_4(4(2^{d}-1))$ is the set consisting of 21 monomials, namely: 

\medskip
\centerline{\begin{tabular}{ll}
$v_{d,1} =  x_1^{2^{d-1}-1}x_2^{2^{d-1}-1}x_3^{2^{d}-1}x_4^{2^{d+1}-1}$& $v_{d,2} =  x_1^{2^{d-1}-1}x_2^{2^{d-1}-1}x_3^{2^{d+1}-1}x_4^{2^{d}-1}$\cr 
$v_{d,3} =  x_1^{2^{d-1}-1}x_2^{2^{d}-1}x_3^{2^{d-1}-1}x_4^{2^{d+1}-1}$& $v_{d,4} =  x_1^{2^{d-1}-1}x_2^{2^{d}-1}x_3^{2^{d+1}-1}x_4^{2^{d-1}-1}$\cr 
$v_{d,5} =  x_1^{2^{d-1}-1}x_2^{2^{d+1}-1}x_3^{2^{d-1}-1}x_4^{2^{d}-1}$& $v_{d,6} =  x_1^{2^{d-1}-1}x_2^{2^{d+1}-1}x_3^{2^{d}-1}x_4^{2^{d-1}-1}$\cr 
$v_{d,7} =  x_1^{2^{d}-1}x_2^{2^{d-1}-1}x_3^{2^{d-1}-1}x_4^{2^{d+1}-1}$& $v_{d,8} =  x_1^{2^{d}-1}x_2^{2^{d-1}-1}x_3^{2^{d+1}-1}x_4^{2^{d-1}-1}$\cr 
$v_{d,9} =  x_1^{2^{d}-1}x_2^{2^{d+1}-1}x_3^{2^{d-1}-1}x_4^{2^{d-1}-1}$& $v_{d,10} =  x_1^{2^{d+1}-1}x_2^{2^{d-1}-1}x_3^{2^{d-1}-1}x_4^{2^{d}-1}$\cr 
$v_{d,11} =  x_1^{2^{d+1}-1}x_2^{2^{d-1}-1}x_3^{2^{d}-1}x_4^{2^{d-1}-1}$& $v_{d,12} =  x_1^{2^{d+1}-1}x_2^{2^{d}-1}x_3^{2^{d-1}-1}x_4^{2^{d-1}-1}$\cr 
$v_{d,13} =  x_1^{2^{d-1}-1}x_2^{2^{d}-1}x_3^{2^{d}-1}x_4^{2^d+2^{d-1}-1}$& $v_{d,14} =  x_1^{2^{d-1}-1}x_2^{2^{d}-1}x_3^{2^d+2^{d-1}-1}x_4^{2^{d}-1}$\cr 
$v_{d,15} =  x_1^{2^{d}-1}x_2^{2^{d-1}-1}x_3^{2^{d}-1}x_4^{2^d+2^{d-1}-1}$& $v_{d,16} =  x_1^{2^{d}-1}x_2^{2^{d-1}-1}x_3^{2^d+2^{d-1}-1}x_4^{2^{d}-1}$\cr 
$v_{d,17} =  x_1^{2^{d}-1}x_2^{2^{d}-1}x_3^{2^{d-1}-1}x_4^{2^d+2^{d-1}-1}$& $v_{d,18} =  x_1^{2^{d}-1}x_2^{2^{d}-1}x_3^{2^d+2^{d-1}-1}x_4^{2^{d-1}-1}$\cr 
$v_{d,19} =  x_1^{2^{d}-1}x_2^{2^d+2^{d-1}-1}x_3^{2^{d-1}-1}x_4^{2^{d}-1}$& $v_{d,20} =  x_1^{2^{d}-1}x_2^{2^d+2^{d-1}-1}x_3^{2^{d}-1}x_4^{2^{d-1}-1}$\cr 
$v_{d,21} =  x_1^{2^{d}-1}x_2^{2^{d}-1}x_3^{2^{d}-1}x_4^{2^{d}-1}$.& 
\end{tabular}}

\medskip
Note that $B_4(4(2^{d}-1)) = B_4(\omega_{(5,d)})\cup B_4(\bar\omega_{(5,d)})$, where
$$ B_4(\omega_{(5,d)}) =\{v_{d,21}\},\ B_4(\bar\omega_{(5,d)}) = \{v_{d,t}: 1 \leqslant t \leqslant 20\}.$$

Since $(P_5)_4 = (P_5^0)_4$, we obtain $B_5(4) = \Phi^0(B_4(4))$, $|B_5(4)| = 45$. It is easy to see that $|\Phi^0(B_4(\bar\omega_{(5,d)}))| = 100$, for $d > 1$. So, we obtain the following.

\begin{prop}\label{mdbsug}
For any integer $d > 1$, $\dim QP_5^0(\bar\omega_{(5,d)}) = 100.$
\end{prop}

\begin{lem}\label{bdday}
If $x$  is an admissible monomial of degree $4(2^{d}-1)$ in  $P_5$, then either
$\omega(x) = \omega_{(5,d)}$ or $\omega(x) = \bar\omega_{(5,d)}.$
\end{lem}
\begin{proof} We prove the lemma by induction on $d$. For $d= 1$, since $x\in B_5(4) = \Phi^0(B_4(4))$,  we have either $\omega(x) = (4, 0)=\omega_{(5,1)}$ or $\omega(x) = (2, 1) = \bar\omega_{(5,1)}$. The lemma holds for $d=1$.

Suppose $d> 1$ and the lemma holds for $1,2,\ldots, d-1$.  Observe that the monomial $z = x_1^{2^{d+1}-1}x_2^{2^{d}-1}x_3^{2^{d-1}-1}x_4^{2^{d-1}-1}$ is the minimal spike of degree $4(2^{d}-1)$ in $P_5$ and  $\omega(z) =  \bar\omega_{(5,d)}.$

 Since $4(2^{d} - 1)$ is even, one gets either $\omega_1(x) = 0$ or $\omega_1(x) = 2$ or $\omega_1(x) = 4.$ If either $\omega_1(x) = 0$ or $\omega_1(x) = 2$ then $\omega(x) < \omega(z).$ By Theorem \ref{dlsig},  $x$ is hit. This contradicts the fact that $x$ is admissible. So, $\omega_1(x) = 4$ and  $x = X_iy^2$ with $1 \leqslant i \leqslant 5$ and $y$ a monomial of degree $4(2^{d-1}-1)$. Since $x$ is admissible, according to Theorem \ref{dlcb1}, $y$ is also admissible. Now, the lemma follows from the inductive hypothesis.
\end{proof}

Since $n= 4(2^d-1) = 2^{d+1} + 2^d + 2^{d-1} + 2^{d-1} -4$, using Theorem \ref{mdc1},
 we get $\dim (QP_5)_{4(2^d-1)} = (2^5-1)\times 21 = 651$ for $d \geqslant 5.$ 

By Lemma \ref{bdday}, we have
$$(QP_5)_{4(2^d-1)} \cong QP_5(\omega_{(5,d)}) \oplus QP_5(\bar\omega_{(5,d)}). $$
Hence, combining Theorem \ref{mdc1} and Proposition \ref{mdcm1}, we obtain $\dim QP_5(\bar\omega_{(5,d)}) = 620$ for $d \geqslant 5$. So, we need only to compute $QP_5(\bar\omega_{(5,d)})$ for $2 \leqslant d \leqslant 4$.

For simplicity, we prove the theorem in detail for the case $d=2$. The others can be proved by a similar computation. The admissible monomials $a_{d,t}$ of degree $4(2^d - 1)$ in $P_5^+$ are explicitly determined in Section 4.

\medskip
\subsection{\it The case $d=2$}\
\setcounter{equation}{0}

\medskip
For $d=2$, we have $4(2^d-1) = 12$. By a direct computation we see that 
$$\Phi^+(B_4(\bar\omega_{(5,2)}))\cup \{x_1^3x_2^4x_3x_4x_5^3,x_1^3x_2^4x_3x_4^3x_5,x_1^3x_2^4x_3^3x_4x_5\}$$
is the set consisting of 75 monomials: $a_{2,t},\ 1 \leqslant t \leqslant 75$ (see Section 4).

\begin{props}\label{props2} The set $\{[a_{2,t}],\ 1 \leqslant t \leqslant 75\}$ is a basis of the $\mathbb F_2$-vector space $QP_5^+(\bar\omega_{(5,2)})$.
\end{props}

The proof of the proposition is based on Theorem \ref{dlcb1} and the following.
\begin{lems}\label{bd11} The following monomials are strictly inadmissible:

\medskip
\rm{(i)} $x_i^2x_jx_\ell x_mx_n^3,\ i < j <\ell <m$,

\rm{(ii)} $x_i^2x_jx_\ell^3 x_m^3x_n^3,\ i < j$.

\medskip
\noindent Here $(i, j,\ell,m,n) $ is a permutation of $(1,2,3,4,5)$.
\end{lems}
\begin{proof} We have
\begin{align*}
x_i^2x_jx_\ell x_mx_n^3 &= x_ix_j^2x_\ell x_mx_n^3 + x_ix_jx_\ell^2 x_mx_n^3 + x_ix_jx_\ell x_m^2x_n^3\\
&\quad + x_ix_jx_\ell x_mx_n^4 + Sq^1(x_ix_jx_\ell x_mx_n^3),\\
x_i^2x_jx_\ell^3 x_m^3x_n^3 &= x_ix_j^2x_\ell^3 x_m^3x_n^3+x_ix_jx_\ell^4 x_m^3x_n^3+x_ix_jx_\ell^3 x_m^4x_n^3\\
&\quad + x_ix_jx_\ell^3 x_m^3x_n^4 + Sq^1(x_ix_jx_\ell^3 x_m^3x_n^3).
\end{align*}
The lemma follows from the above equalities.
\end{proof}
\begin{lems}\label{bd12} The $\mathbb F_2$-vector space $QP_5^+(\bar\omega_{(5,2)})$ is spanned by the set 
$$\{[a_{2,t}],\ 1 \leqslant t \leqslant 75\}.$$
\end{lems}
\begin{proof} Let $x$ be an admissible monomial in $P_5$ such that $\omega(x) = \bar\omega_{(5,2)}$. Then, $\omega_1(x)=4$, $x = X_jy^2$ with $1\leqslant j \leqslant 5$ and $y$ a monomial of degree 4 in $P_5$. Since $x$ is admissible, according to Theorem \ref{dlcb1}, $y \in B_5(\bar\omega_{(5,1)})$. 

Let $z \in B_5(\bar\omega_{(5,1)})$ and $1\leqslant j \leqslant 5$. 
By a direct computation we see that if $X_jz^2\ne a_{2,t}, \forall t, 1 \leqslant t \leqslant 75$, then there is a monomial $w$ which is given in Lemma \ref{bd11} such that $X_jz^2= wz_1^{2^{u}}$ with suitable monomial $z_1 \in P_5$, and $u = \max\{s \in \mathbb Z : \omega_s(w) >0\}$. By Theorem \ref{dlcb1}, $X_jz^2$ is inadmissible. Since $x = X_jy^2$ with  $y \in B_5(\bar\omega_{(5,1)})$ and $x$ is admissible, one gets $x = a_{2, t}$ for some $t$. The lemma is proved.
\end{proof}
\begin{lems}\label{bd13} The set  $\{[a_{2,t}],\ 1 \leqslant t \leqslant 75\}$ is linearly independent in $QP_5^+(\bar\omega_{(5,2)})$.
\end{lems}
\begin{proof} Suppose there is a linear relation
\begin{equation*}\label{cts20} \mathcal S = \sum_{t=1}^{75}\gamma_ta_{2,t} \equiv 0, 
\end{equation*}
where $\gamma_t\in \mathbb F_2$ for all $t,\ 1 \leqslant t \leqslant 75$.

Let $J$ be a sequence of non-negative integers and $\gamma_j \in \mathbb F_2$ for $ j\in J$. Denote by $\gamma_J = \sum_{j\in J}\gamma_j \in \mathbb F_2.$ Based on Theorem \ref{dlsig}, for $(i;I) \in \mathcal N_5$, we explicitly compute $p_{(i;I)}(\mathcal S)$ in terms of $v_{2,j},\ 1 \leqslant j \leqslant 20$. By computing from the relations $p_{(i;j)}(\mathcal S) \equiv 0,\ 1 \leqslant i < j \leqslant 5$, we have 
\begin{equation}\label{cts21}
 \begin{cases}\gamma_t =0,\ t \in \mathbb J,\ \ \gamma_{55}  = \gamma_{56} = \gamma_{57}, \\
\gamma_t = \gamma_{4},\ t = 15, 26, 65, 68, 73,\\ 
\gamma_t = \gamma_{5},\ t = 16, 34, 64, 69, 74,\\
\gamma_t = \gamma_{6},\ t = 27, 36, 66, 67, 75,\\ 
\gamma_t = \gamma_{17},\ t = 29, 38, 70, 71, 72, \\
\gamma_{\{6,17,43,55\}} =  \gamma_{\{5,17,47,55\}} = 0,\\  
\gamma_{\{4,17,51,55\}} =  \gamma_{\{5,6,43,47\}} = 0,\\ 
\gamma_{\{4,6,43,51\}} =  \gamma_{\{4,5,47,51\}} =  0.
\end{cases}
\end{equation}
Here $\mathbb J= \{$1, 2, 3, 7, 8, 9, 10, 11, 12, 13, 14, 18, 19, 20, 21, 22, 23, 24, 25, 28, 30, 31, 32, 33, 35, 37, 39, 40, 41, 42, 44, 45, 46, 48, 49, 50, 52, 53, 54, 58, 59, 60, 61, 62, 63$\}.$ Then, computing from the relations $p_{(1;(2,3))}(\mathcal S) \equiv 0$ and $p_{(1;(2,4))}(\mathcal S) \equiv 0$ gives
\begin{equation}\label{cts23}
 \begin{cases}\gamma_t =0,\ t \in \{4, 5, 6, 43, 47, 55\},\\
\gamma_{\{4,5,47,51\}} =0.
\end{cases}
\end{equation}
By combining the relations (\ref{cts21}) and (\ref{cts23}), we obtain $\gamma_t =0$, $1 \leqslant t \leqslant 75$. The lemma is proved.
\end{proof}

From Propositions \ref{mdcm1}, \ref{mdbsug} and \ref{props2}, we have $\dim (QP_5)_{12} = 190.$

\subsection{\it The case $d=3$}\
\setcounter{equation}{0} 

\medskip
 For $d=3$, we have $4(2^{d}-1) = 28.$ Denote by $C$ the set of the following monomials:

\medskip
\centerline{\begin{tabular}{lll}
$x_1^{7}x_2^{9}x_3^{2}x_4^{3}x_5^{7}$& $x_1^{7}x_2^{9}x_3^{2}x_4^{7}x_5^{3}$& $x_1^{7}x_2^{9}x_3^{3}x_4^{2}x_5^{7}$\cr 
$x_1^{7}x_2^{9}x_3^{3}x_4^{7}x_5^{2}$& $x_1^{7}x_2^{9}x_3^{7}x_4^{2}x_5^{3}$ & $x_1^{7}x_2^{9}x_3^{7}x_4^{3}x_5^{2}$\cr 
$x_1^{3}x_2^{7}x_3^{11}x_4^{4}x_5^{3}$& $x_1^{7}x_2^{3}x_3^{11}x_4^{4}x_5^{3}$& $x_1^{7}x_2^{11}x_3^{3}x_4^{4}x_5^{3}$\cr 
$x_1^{7}x_2^{9}x_3^{3}x_4^{3}x_5^{6}$& $x_1^{7}x_2^{9}x_3^{3}x_4^{6}x_5^{3}$& $x_1^{3}x_2^{7}x_3^{7}x_4^{8}x_5^{3}$\cr 
$x_1^{3}x_2^{7}x_3^{8}x_4^{3}x_5^{7}$& $x_1^{3}x_2^{7}x_3^{8}x_4^{7}x_5^{3}$& $x_1^{7}x_2^{3}x_3^{7}x_4^{8}x_5^{3}$\cr 
$x_1^{7}x_2^{3}x_3^{8}x_4^{3}x_5^{7}$& $x_1^{7}x_2^{3}x_3^{8}x_4^{7}x_5^{3}$& $x_1^{7}x_2^{7}x_3^{3}x_4^{8}x_5^{3}$\cr 
$x_1^{7}x_2^{7}x_3^{8}x_4^{3}x_5^{3}$.&&\cr
\end{tabular}}

\medskip
A direct computation shows that $\Phi^+(B_4(\bar\omega_{(5,3)}))\cup C$ is the set of $355$ monomials: $a_{3,t}$, $1\leqslant t \leqslant 355$ (see Section 4).

\begin{props}\label{props3} Under the above notations, the set $[\Phi^+(B_4(\bar\omega_{(5,3)}))\cup C]$ is a basis of the $\mathbb F_2$-vector space  $QP_5^+(\bar\omega_{(5,3)})$.
\end{props}
 
We prove the proposition by proving some lemmas.
\begin{lems}\label{bd41} If $(i, j,\ell,m,n) $ is a permutation of $(1,2,3,4,5)$ such that $i<j$, then the monomial $ x_i^3x_j^4x_\ell^7x_m^7x_n^7$ is strictly inadmissible.
\end{lems} 
\begin{proof} By a direct computation, we have
\begin{align*}
x_i^3x_j^4x_\ell^7x_m^7x_n^7&= x_i^2x_j^5x_\ell^7 x_m^7x_n^7 + x_i^2x_j^3x_\ell^9 x_m^7x_n^7\\
&\quad + x_i^2x_j^3x_\ell^7 x_m^9x_n^7 + x_i^2x_j^3x_\ell^7 x_m^7x_n^9\\
&\quad  + x_i^3x_j^3x_\ell^8 x_m^7x_n^7 + x_i^3x_j^3x_\ell^7 x_m^8x_n^7 \\
&\quad + x_i^3x_j^3x_\ell^7 x_m^7x_n^8 + x_i^2x_j^3x_\ell^8 x_m^8x_n^7\\
&\quad + x_i^2x_j^3x_\ell^8 x_m^7x_n^8   + x_i^2x_j^4x_\ell^7 x_m^8x_n^8\\
&\quad +  Sq^1(x_i^3x_j^3x_\ell^7 x_m^7x_n^7+ x_i^2x_j^4x_\ell^7 x_m^7x_n^7 )\\
&\quad + Sq^2(x_i^2x_j^3x_\ell^7 x_m^7x_n^7). 
\end{align*}
The lemma is proved.
\end{proof}

By a similar computation, one gets the following.
\begin{lems}\label{bd42} The following monomials are strictly inadmissible:

\medskip
\rm{(i)} $x_i^3x_j^4x_\ell^3x_m^3x_n^7,\  i < j <4,\ 2<\ell < m$, 

\rm{(ii)} $x_1^{7}x_2^{9}x_3^{6}x_4^{3}x_5^{3}$, $x_1^{7}x_2^{8}x_3^{3}x_4^{3}x_5^{7}$, $x_1^{7}x_2^{8}x_3^{3}x_4^{7}x_5^{3}$, $x_1^{7}x_2^{8}x_3^{7}x_4^{3}x_5^{3}$.

\medskip\noindent
Here $(i, j,\ell,m,n) $ is a permutation of $(1,2,3,4,5)$.
\end{lems}

\begin{lems}\label{bd43} The $\mathbb F_2$-vector space $QP_5^+(\bar\omega_{(5,3)})$ is spanned by the set 
$$[\Phi^+(B_4(\bar\omega_{(5,3)}))\cup C].$$
\end{lems}
\begin{proof} Let $x$ be an admissible monomial in $P_5$ such that $\omega(x) = \bar\omega_{(5,3)}$. Then,  $\omega_1(x) = 4$ and $x = X_jy^2$ with $1\leqslant j \leqslant 5$ and $y$ a monomial of degree 12 in $P_5$. Since $x$ is admissible, according to Theorem \ref{dlcb1}, $y \in B_5(\bar\omega_{(5,2)})$. 

A direct computation shows that if $z \in B_5(\bar\omega_{(5,2)})$, $1\leqslant j \leqslant 5$ and $X_jz^2\ne a_{3,t}, \forall t, 1\leqslant t \leqslant 355$, then there exists a monomial $w$ which is given in one of  Lemmas \ref{bd11}(ii), \ref{bd41}, \ref{bd42} such that $X_jz^2= wz_1^{2^{r}}$ with a monomial $z_1 \in P_5$, and $r = \max\{s \in \mathbb Z : \omega_s(w) >0\}$. By Theorem \ref{dlcb1}, $X_jz^2$ is inadmissible. Since $x = X_jy^2$ with $y \in B_5(\bar\omega_{(5,2)})$ and $x$ is admissible, one can see that  $x =  a_{3,t}$ for suitable $t$. The lemma follows.
\end{proof}
\begin{lems}\label{bd44} The set  $[\Phi^+(B_4(\bar\omega_{(5,3)}))\cup C]$ is linearly independent in $QP_5^+(\bar\omega_{(5,3)})$.
\end{lems}
\begin{proof} Suppose there is a linear relation
$$ \mathcal S = \sum_{t=1}^{355}\gamma_ta_{3,t} \equiv 0, $$
where $\gamma_t \in \mathbb F_2$ for all $1 \leqslant t \leqslant 355$. Using Theorem \ref{dlsig}, we explicitly compute $ p_{(i;I)}(S)$, $(i;I)\in \mathcal N_5,$ in terms of $v_{3,j},\ 1 \leqslant j \leqslant 20$. By a direct computation from the relations $ p_{(i;I)}(S) \equiv 0$, $(i;I) \in \mathcal N_5$ with $0<\ell(I) \leqslant 3$, we obtain $\gamma_t =0$ for $1 \leqslant t \leqslant 355$.
\end{proof}

By using Propositions \ref{mdcm1}, \ref{mdbsug} and \ref{props3}, we get $\dim (QP_5)_{28} = 480.$  

\medskip
\subsection{\it The case $d=4$}\
\setcounter{equation}{0} 

\medskip
For $d=4$, we have $4(2^{d}-1) = 60.$ 
Denote by $D$ the set of the following monomials:

\medskip
\centerline{\begin{tabular}{lll}
$x_1^{7}x_2^{7}x_3^{15}x_4^{15}x_5^{16} $& $x_1^{7}x_2^{7}x_3^{15}x_4^{23}x_5^{8} $& $x_1^{7}x_2^{15}x_3^{7}x_4^{15}x_5^{16} $\cr 
$x_1^{7}x_2^{15}x_3^{7}x_4^{23}x_5^{8} $& $x_1^{7}x_2^{15}x_3^{15}x_4^{7}x_5^{16} $& $x_1^{7}x_2^{15}x_3^{15}x_4^{16}x_5^{7} $\cr 
$x_1^{7}x_2^{15}x_3^{15}x_4^{17}x_5^{6} $& $x_1^{7}x_2^{15}x_3^{17}x_4^{7}x_5^{14} $& $x_1^{7}x_2^{15}x_3^{23}x_4^{7}x_5^{8} $\cr 
$x_1^{7}x_2^{15}x_3^{23}x_4^{9}x_5^{6} $& $x_1^{15}x_2^{7}x_3^{7}x_4^{15}x_5^{16} $& $x_1^{15}x_2^{7}x_3^{7}x_4^{23}x_5^{8} $\cr  
$x_1^{15}x_2^{7}x_3^{15}x_4^{7}x_5^{16} $& $x_1^{15}x_2^{7}x_3^{15}x_4^{16}x_5^{7} $& $x_1^{15}x_2^{7}x_3^{15}x_4^{17}x_5^{6} $\cr 
$x_1^{15}x_2^{7}x_3^{17}x_4^{7}x_5^{14} $&  $x_1^{15}x_2^{7}x_3^{23}x_4^{7}x_5^{8} $& $x_1^{15}x_2^{7}x_3^{23}x_4^{9}x_5^{6} $\cr 
$x_1^{15}x_2^{15}x_3^{7}x_4^{7}x_5^{16} $& $x_1^{15}x_2^{15}x_3^{7}x_4^{16}x_5^{7} $& $x_1^{15}x_2^{15}x_3^{7}x_4^{17}x_5^{6} $\cr 
$x_1^{15}x_2^{15}x_3^{17}x_4^{6}x_5^{7} $& $x_1^{15}x_2^{15}x_3^{17}x_4^{7}x_5^{6} $& $x_1^{15}x_2^{19}x_3^{5}x_4^{7}x_5^{14} $\cr  
$x_1^{15}x_2^{19}x_3^{7}x_4^{5}x_5^{14} $& $x_1^{15}x_2^{19}x_3^{7}x_4^{13}x_5^{6} $& $x_1^{15}x_2^{23}x_3^{7}x_4^{7}x_5^{8} $\cr 
$x_1^{15}x_2^{23}x_3^{7}x_4^{9}x_5^{6} $.& &\cr  
\end{tabular}}

\medskip
A direct computation shows that $\Phi^+(B_4(\bar\omega_{(5,4)}))\cup D$ is the set of 520 monomials: $ a_{4,t},\ 1 \leqslant t \leqslant 520$ (see Section 4).

\begin{props}\label{prop4} Under the above notation, the set $[\Phi^+(B_4(\bar\omega_{(5,4)}))\cup D]$ is a basis of the $\mathbb F_2$-vector space $QP_5^+(\bar\omega_{(5,4)})$.
\end{props}

We prepare some lemmas for the proof of this proposition. The following lemma is proved by a direct computation.
\begin{lems}\label{bd51} If $(i, j,\ell,m,n) $ is a permutation of $(1,2,3,4,5)$ such that $ i < j < \ell < m$, then the monomial 
$x_i^7x_j^7x_\ell^8x_m^7x_n^{15}$ is strictly inadmissible.
\end{lems}

\begin{lems}\label{bd52} The following monomials are strictly inadmissible:

\medskip
\centerline{\begin{tabular}{lll}
$x_1^{7}x_2^{7}x_3^{15}x_4^{16}x_5^{15}$& $x_1^{7}x_2^{15}x_3^{7}x_4^{16}x_5^{15}$& $x_1^{7}x_2^{15}x_3^{16}x_4^{7}x_5^{15}$\cr 
$x_1^{7}x_2^{15}x_3^{16}x_4^{15}x_5^{7}$&$x_1^{7}x_2^{15}x_3^{17}x_4^{6}x_5^{15}$& $x_1^{7}x_2^{15}x_3^{17}x_4^{14}x_5^{7}$\cr 
$x_1^{7}x_2^{15}x_3^{17}x_4^{15}x_5^{6}$& $x_1^{15}x_2^{7}x_3^{7}x_4^{16}x_5^{15}$& $x_1^{15}x_2^{7}x_3^{16}x_4^{7}x_5^{15}$\cr 
$x_1^{15}x_2^{7}x_3^{16}x_4^{15}x_5^{7}$& $x_1^{15}x_2^{7}x_3^{17}x_4^{6}x_5^{15}$& $x_1^{15}x_2^{7}x_3^{17}x_4^{14}x_5^{7}$\cr 
$x_1^{15}x_2^{7}x_3^{17}x_4^{15}x_5^{6}$& $x_1^{15}x_2^{15}x_3^{16}x_4^{7}x_5^{7}$& $x_1^{15}x_2^{19}x_3^{5}x_4^{6}x_5^{15}$\cr 
$x_1^{15}x_2^{19}x_3^{5}x_4^{14}x_5^{7}$& $x_1^{15}x_2^{19}x_3^{5}x_4^{15}x_5^{6}$& $x_1^{15}x_2^{19}x_3^{15}x_4^{5}x_5^{6}$.\cr 
\end{tabular}}
\end{lems}
\begin{proof} We prove the lemma for $x=x_1^{7}x_2^{7}x_3^{15}x_4^{16}x_5^{15}$. The others can be prove by a similar computation.

We have $\omega(x) = \bar\omega_{(5,4)}$.  By a direct computation using the Cartan formula, we have
\begin{align*}
x &= x_1^{7}x_2^{7}x_3^{15}x_4^{15}x_5^{16} + x_1^{7}x_2^{7}x_3^{8}x_4^{23}x_5^{15} + x_1^{7}x_2^{7}x_3^{8}x_4^{15}x_5^{23}\\ 
&\quad + x_1^{4}x_2^{7}x_3^{19}x_4^{15}x_5^{15} + x_1^{4}x_2^{7}x_3^{15}x_4^{19}x_5^{15} + x_1^{4}x_2^{7}x_3^{15}x_4^{15}x_5^{19}\\ 
&\quad + x_1^{7}x_2^{6}x_3^{17}x_4^{15}x_5^{15} + x_1^{7}x_2^{6}x_3^{15}x_4^{17}x_5^{15} + x_1^{7}x_2^{6}x_3^{15}x_4^{15}x_5^{17}\\ 
&\quad  + x_1^{5}x_2^{6}x_3^{19}x_4^{15}x_5^{15} + x_1^{5}x_2^{6}x_3^{15}x_4^{19}x_5^{15} + x_1^{5}x_2^{6}x_3^{15}x_4^{15}x_5^{19}\\ 
&\quad + Sq^1(x_1^{7}x_2^{7}x_3^{15}x_4^{15}x_5^{15} + x_1^{3}x_2^{11}x_3^{15}x_4^{15}x_5^{15})\\
&\quad  + Sq^2(x_1^{7}x_2^{6}x_3^{15}x_4^{15}x_5^{15} + x_1^{3}x_2^{10}x_3^{15}x_4^{15}x_5^{15})\\
&\quad +Sq^4(x_1^{4}x_2^{7}x_3^{15}x_4^{15}x_5^{15} + x_1^{5}x_2^{6}x_3^{15}x_4^{15}x_5^{15})
\\ &\quad + Sq^8(x_1^{7}x_2^{7}x_3^{8}x_4^{15}x_5^{15}) \ \  \text{mod}(P_5^-(\bar\omega_{(5,4)})).
\end{align*}
Hence, $x$ is strictly inadmissible.
\end{proof}
\begin{lems}\label{bd53} The $\mathbb F_2$-vector space $QP_5^+(\bar\omega_{(5,4)})$ is spanned by the set 
$$[\Phi^+(B_4(\bar\omega_{(5,4)}))\cup D].$$
\end{lems}
\begin{proof} Let $x$ be an admissible monomial in $P_5$ such that  $\omega(x) = \bar \omega_{(5,4)}$. Then, $\omega_1(x) = 4$, and $x = X_jy^2$ with $1\leqslant j \leqslant 5$ and $y$ a monomial in $P_5$ such that  $\omega(y) = \bar \omega_{(5,3)}$. Since $x$ is admissible, according to Theorem \ref{dlcb1}, we have $y \in B_5(\bar \omega_{(5,3)})$. 

By a direct computation we can verify that for any $z \in B_5(\bar \omega_{(5,3)})$, $1\leqslant j \leqslant 5$, such that $X_jz^2\ne a_{4,t}, \forall t, \ 1\leqslant t \leqslant 520$, there is a monomial $w$ which is given in one of  Lemmas \ref{bd11}(ii), \ref{bd41}, \ref{bd51}, \ref{bd52} such that $X_jz^2= wu^{2^{r}}$ with a monomial $u \in P_5$, and $r = \max\{s \in \mathbb Z : \omega_s(w) >0\}$. By Theorem \ref{dlcb1}, $X_jz^2$ is inadmissible. Since $x = X_jy^2$ is admissible and $y \in B_5(\bar \omega_{(5,3)})$, one gets $x = a_{4,t}$ for some $t$. This proves the lemma.
\end{proof}
\begin{lems}\label{bd54} The set  $[\Phi^+(B_4(\bar\omega_{(5,4)}))\cup D]$ is linearly independent in the $\mathbb F_2$-vector space $QP_5^+(\bar\omega_{(5,4)})$.
\end{lems}
\begin{proof} Suppose there is a linear relation
$$ \mathcal S = \sum_{t=1}^{520}\gamma_ta_{4,t} \equiv 0, $$
where $\gamma_t \in \mathbb F_2$ for all $t,\ 1\leqslant t \leqslant 520$. By a direct computation using Theorem \ref{dlsig}, we express $p_{(i;I)}(\mathcal S)$ in terms of $v_{4,j}, 1 \leqslant j \leqslant 20$.  Computing directly from the relations 
$$ p_{(i,I)}(S) \equiv 0, \ \forall (i;I) \in \mathcal N_5,\ \ell(I)>0,$$
we get $\gamma_t =0$ for all $t$. The lemma is proved.
\end{proof}

By using Propositions \ref{mdcm1}, \ref{mdbsug} and \ref{prop4}, we get $\dim (QP_5)_{60} = 650.$ The Main Theorem is completely proved.

\section{The admissible monomials of degree $4(2^d-1)$ in $P_5$}\label{s4}

In this section, we list all the admissible monomials of degree $4(2^d-1)$ in $P_5$. Recall that
$$B_5(4(2^d-1)) = \Phi^0(B_4(4(2^d-1))) \cup B_5^+(\omega_{(5,d)}) \cup B_5^+(\bar\omega_{(5,d)}),$$
where $B_4(4(2^d-1)) = \{v_{d,j} : 1 \leqslant j \leqslant 21\}$ and 
$$B_5^+(\omega_{(5,d)}) = \{\phi_{(i;I)}(v_{d,21}) : (i;I) \in \mathcal N_5, 1\leqslant \ell(I) < \min\{5,d\}\}.$$
Set $b_d = |B_5^+(\bar\omega_{(5,d)})|$ and $B_5^+(\bar\omega_{(5,d)}) = \{a_{d,t}: 1 \leqslant t \leqslant b_d\}$. We have $b_1 = 0, b_2 = 75, b_3 = 355$ and $b_d = 520$ for $d \geqslant 4$. The admissible monomials $a_{d,t},\ 1 \leqslant t \leqslant b_d,$ are determined as follows:

\medskip
For $d \geqslant 2$,

\medskip
 \centerline{\begin{tabular}{llll}
1. & $x_1x_2^{2^{d-1}-1}x_3^{2^{d-1}-1}x_4^{2^d-2}x_5^{2^{d+1}-1} $& 2. & $x_1x_2^{2^{d-1}-1}x_3^{2^{d-1}-1}x_4^{2^d-1}x_5^{2^{d+1}-2} $\cr 
3. & $x_1x_2^{2^{d-1}-1}x_3^{2^{d-1}-1}x_4^{2^{d+1}-2}x_5^{2^d-1} $& 4. & $x_1x_2^{2^{d-1}-1}x_3^{2^{d-1}-1}x_4^{2^{d+1}-1}x_5^{2^d-2} $\cr 
5. & $x_1x_2^{2^{d-1}-1}x_3^{2^d-2}x_4^{2^{d-1}-1}x_5^{2^{d+1}-1} $& 6. & $x_1x_2^{2^{d-1}-1}x_3^{2^d-2}x_4^{2^d-1}x_5^{2^{d+1}-2^{d-1}-1} $\cr 
7. & $x_1x_2^{2^{d-1}-1}x_3^{2^d-2}x_4^{2^{d+1}-2^{d-1}-1}x_5^{2^d-1} $& 8. & $x_1x_2^{2^{d-1}-1}x_3^{2^d-2}x_4^{2^{d+1}-1}x_5^{2^{d-1}-1} $\cr 
9. & $x_1x_2^{2^{d-1}-1}x_3^{2^d-1}x_4^{2^{d-1}-1}x_5^{2^{d+1}-2} $& 10. & $x_1x_2^{2^{d-1}-1}x_3^{2^d-1}x_4^{2^d-2}x_5^{2^{d+1}-2^{d-1}-1} $\cr 
11. & $x_1x_2^{2^{d-1}-1}x_3^{2^d-1}x_4^{2^{d+1}-2^{d-1}-1}x_5^{2^d-2} $& 12. & $x_1x_2^{2^{d-1}-1}x_3^{2^d-1}x_4^{2^{d+1}-2}x_5^{2^{d-1}-1} $\cr 
 \end{tabular}}
 \centerline{\begin{tabular}{llll}
13. & $x_1x_2^{2^{d-1}-1}x_3^{2^{d+1}-2}x_4^{2^{d-1}-1}x_5^{2^d-1} $& 14. & $x_1x_2^{2^{d-1}-1}x_3^{2^{d+1}-2}x_4^{2^d-1}x_5^{2^{d-1}-1} $\cr 
15. & $x_1x_2^{2^{d-1}-1}x_3^{2^{d+1}-1}x_4^{2^{d-1}-1}x_5^{2^d-2} $& 16. & $x_1x_2^{2^{d-1}-1}x_3^{2^{d+1}-1}x_4^{2^d-2}x_5^{2^{d-1}-1} $\cr 
17. & $x_1x_2^{2^d-2}x_3^{2^{d-1}-1}x_4^{2^{d-1}-1}x_5^{2^{d+1}-1} $& 18. & $x_1x_2^{2^d-2}x_3^{2^{d-1}-1}x_4^{2^d-1}x_5^{2^{d+1}-2^{d-1}-1} $\cr 
19. & $x_1x_2^{2^d-2}x_3^{2^{d-1}-1}x_4^{2^{d+1}-2^{d-1}-1}x_5^{2^d-1} $& 20. & $x_1x_2^{2^d-2}x_3^{2^{d-1}-1}x_4^{2^{d+1}-1}x_5^{2^{d-1}-1} $\cr 
21. & $x_1x_2^{2^d-2}x_3^{2^d-1}x_4^{2^{d-1}-1}x_5^{2^{d+1}-2^{d-1}-1} $& 22. & $x_1x_2^{2^d-2}x_3^{2^d-1}x_4^{2^{d+1}-2^{d-1}-1}x_5^{2^{d-1}-1} $\cr 
23. & $x_1x_2^{2^d-2}x_3^{2^{d+1}-2^{d-1}-1}x_4^{2^{d-1}-1}x_5^{2^d-1} $& 24. & $x_1x_2^{2^d-2}x_3^{2^{d+1}-2^{d-1}-1}x_4^{2^d-1}x_5^{2^{d-1}-1} $\cr 
25. & $x_1x_2^{2^d-2}x_3^{2^{d+1}-1}x_4^{2^{d-1}-1}x_5^{2^{d-1}-1} $& 26. & $x_1x_2^{2^d-1}x_3^{2^{d-1}-1}x_4^{2^{d-1}-1}x_5^{2^{d+1}-2} $\cr 
27. & $x_1x_2^{2^d-1}x_3^{2^{d-1}-1}x_4^{2^d-2}x_5^{2^{d+1}-2^{d-1}-1} $& 28. & $x_1x_2^{2^d-1}x_3^{2^{d-1}-1}x_4^{2^{d+1}-2^{d-1}-1}x_5^{2^d-2} $\cr 
29. & $x_1x_2^{2^d-1}x_3^{2^{d-1}-1}x_4^{2^{d+1}-2}x_5^{2^{d-1}-1} $& 30. & $x_1x_2^{2^d-1}x_3^{2^d-2}x_4^{2^{d-1}-1}x_5^{2^{d+1}-2^{d-1}-1} $\cr 
31. & $x_1x_2^{2^d-1}x_3^{2^d-2}x_4^{2^{d+1}-2^{d-1}-1}x_5^{2^{d-1}-1} $& 32. & $x_1x_2^{2^d-1}x_3^{2^{d+1}-2^{d-1}-2}x_4^{2^{d-1}-1}x_5^{2^d-1} $\cr 
33. & $x_1x_2^{2^d-1}x_3^{2^{d+1}-2^{d-1}-2}x_4^{2^d-1}x_5^{2^{d-1}-1} $& 34. & $x_1x_2^{2^d-1}x_3^{2^{d+1}-2^{d-1}-1}x_4^{2^{d-1}-1}x_5^{2^d-2} $\cr 
35. & $x_1x_2^{2^d-1}x_3^{2^{d+1}-2^{d-1}-1}x_4^{2^d-2}x_5^{2^{d-1}-1} $& 36. & $x_1x_2^{2^d-1}x_3^{2^{d+1}-2}x_4^{2^{d-1}-1}x_5^{2^{d-1}-1} $\cr 
37. & $x_1x_2^{2^{d+1}-2}x_3^{2^{d-1}-1}x_4^{2^{d-1}-1}x_5^{2^d-1} $& 38. & $x_1x_2^{2^{d+1}-2}x_3^{2^{d-1}-1}x_4^{2^d-1}x_5^{2^{d-1}-1} $\cr 
39. & $x_1x_2^{2^{d+1}-2}x_3^{2^d-1}x_4^{2^{d-1}-1}x_5^{2^{d-1}-1} $& 40. & $x_1x_2^{2^{d+1}-1}x_3^{2^{d-1}-1}x_4^{2^{d-1}-1}x_5^{2^d-2} $\cr 
41. & $x_1x_2^{2^{d+1}-1}x_3^{2^{d-1}-1}x_4^{2^d-2}x_5^{2^{d-1}-1} $& 42. & $x_1x_2^{2^{d+1}-1}x_3^{2^d-2}x_4^{2^{d-1}-1}x_5^{2^{d-1}-1} $\cr 
43. & $x_1^{3}x_2^{2^{d+1}-3}x_3^{2^d-2}x_4^{2^{d-1}-1}x_5^{2^{d-1}-1} $& 44. & $x_1^{2^d-1}x_2x_3^{2^{d-1}-1}x_4^{2^{d-1}-1}x_5^{2^{d+1}-2} $\cr 
45. & $x_1^{2^d-1}x_2x_3^{2^{d-1}-1}x_4^{2^d-2}x_5^{2^{d+1}-2^{d-1}-1} $& 46. & $x_1^{2^d-1}x_2x_3^{2^{d-1}-1}x_4^{2^{d+1}-2^{d-1}-1}x_5^{2^d-2} $\cr 
47. & $x_1^{2^d-1}x_2x_3^{2^{d-1}-1}x_4^{2^{d+1}-2}x_5^{2^{d-1}-1} $& 48. & $x_1^{2^d-1}x_2x_3^{2^d-2}x_4^{2^{d-1}-1}x_5^{2^{d+1}-2^{d-1}-1} $\cr 
49. & $x_1^{2^d-1}x_2x_3^{2^d-2}x_4^{2^{d+1}-2^{d-1}-1}x_5^{2^{d-1}-1} $& 50. & $x_1^{2^d-1}x_2x_3^{2^{d+1}-2^{d-1}-1}x_4^{2^{d-1}-1}x_5^{2^d-2} $\cr 
51. & $x_1^{2^d-1}x_2x_3^{2^{d+1}-2^{d-1}-1}x_4^{2^d-2}x_5^{2^{d-1}-1} $& 52. & $x_1^{2^d-1}x_2x_3^{2^{d+1}-2}x_4^{2^{d-1}-1}x_5^{2^{d-1}-1} $\cr 
53. & $x_1^{2^d-1}x_2^{2^{d+1}-2^{d-1}-1}x_3x_4^{2^{d-1}-1}x_5^{2^d-2} $& 54. & $x_1^{2^d-1}x_2^{2^{d+1}-2^{d-1}-1}x_3x_4^{2^d-2}x_5^{2^{d-1}-1} $\cr 
55. & $x_1^{2^{d+1}-1}x_2x_3^{2^{d-1}-1}x_4^{2^{d-1}-1}x_5^{2^d-2} $& 56. & $x_1^{2^{d+1}-1}x_2x_3^{2^{d-1}-1}x_4^{2^d-2}x_5^{2^{d-1}-1} $\cr 
57. & $x_1^{2^{d+1}-1}x_2x_3^{2^d-2}x_4^{2^{d-1}-1}x_5^{2^{d-1}-1} $&&\cr
 \end{tabular}}

\medskip 
For $d=2$,

\medskip
 \centerline{\begin{tabular}{llll}
58. \ $x_1x_2x_3^{3}x_4^{3}x_5^{4} $& 59. \ $x_1x_2x_3^{3}x_4^{4}x_5^{3} $& 60. \ $x_1x_2^{3}x_3x_4^{3}x_5^{4} $& 61. \ $x_1x_2^{3}x_3x_4^{4}x_5^{3} $\cr  
62. \ $x_1x_2^{3}x_3^{3}x_4x_5^{4} $& 63. \ $x_1x_2^{3}x_3^{3}x_4^{4}x_5 $& 64. \ $x_1^{3}x_2x_3x_4^{3}x_5^{4} $& 65. \ $x_1^{3}x_2x_3x_4^{4}x_5^{3} $\cr  
66. \ $x_1^{3}x_2x_3^{3}x_4x_5^{4} $& 67. \ $x_1^{3}x_2x_3^{3}x_4^{4}x_5 $& 68. \ $x_1^{3}x_2x_3^{4}x_4x_5^{3} $& 69. \ $x_1^{3}x_2x_3^{4}x_4^{3}x_5 $\cr  
70. \ $x_1^{3}x_2^{3}x_3x_4x_5^{4} $& 71. \ $x_1^{3}x_2^{3}x_3x_4^{4}x_5 $& 72. \ $x_1^{3}x_2^{3}x_3^{4}x_4x_5 $& 73. \ $x_1^{3}x_2^{4}x_3x_4x_5^{3} $\cr  
74. \ $x_1^{3}x_2^{4}x_3x_4^{3}x_5 $& 75. \ $x_1^{3}x_2^{4}x_3^{3}x_4x_5 $& 
 \end{tabular}}

\medskip
For $d \geqslant 3$,

\medskip
 \centerline{\begin{tabular}{llll}
58. & $x_1x_2^{2^{d-1}-2}x_3^{2^{d-1}-1}x_4^{2^d-1}x_5^{2^{d+1}-1} $& 59. & $x_1x_2^{2^{d-1}-2}x_3^{2^{d-1}-1}x_4^{2^{d+1}-1}x_5^{2^d-1} $\cr 
60. & $x_1x_2^{2^{d-1}-2}x_3^{2^d-1}x_4^{2^{d-1}-1}x_5^{2^{d+1}-1} $& 61. & $x_1x_2^{2^{d-1}-2}x_3^{2^d-1}x_4^{2^d-1}x_5^{2^{d+1}-2^{d-1}-1} $\cr 
62. & $x_1x_2^{2^{d-1}-2}x_3^{2^d-1}x_4^{2^{d+1}-2^{d-1}-1}x_5^{2^d-1} $& 63. & $x_1x_2^{2^{d-1}-2}x_3^{2^d-1}x_4^{2^{d+1}-1}x_5^{2^{d-1}-1} $\cr 
64. & $x_1x_2^{2^{d-1}-2}x_3^{2^{d+1}-1}x_4^{2^{d-1}-1}x_5^{2^d-1} $& 65. & $x_1x_2^{2^{d-1}-2}x_3^{2^{d+1}-1}x_4^{2^d-1}x_5^{2^{d-1}-1} $\cr 
66. & $x_1x_2^{2^{d-1}-1}x_3^{2^{d-1}-2}x_4^{2^d-1}x_5^{2^{d+1}-1} $& 67. & $x_1x_2^{2^{d-1}-1}x_3^{2^{d-1}-2}x_4^{2^{d+1}-1}x_5^{2^d-1} $\cr 
68. & $x_1x_2^{2^{d-1}-1}x_3^{2^d-1}x_4^{2^{d-1}-2}x_5^{2^{d+1}-1} $& 69. & $x_1x_2^{2^{d-1}-1}x_3^{2^d-1}x_4^{2^d-1}x_5^{2^{d+1}-2^{d-1}-2} $\cr 
70. & $x_1x_2^{2^{d-1}-1}x_3^{2^d-1}x_4^{2^{d+1}-2^{d-1}-2}x_5^{2^d-1} $& 71. & $x_1x_2^{2^{d-1}-1}x_3^{2^d-1}x_4^{2^{d+1}-1}x_5^{2^{d-1}-2} $\cr 
72. & $x_1x_2^{2^{d-1}-1}x_3^{2^{d+1}-1}x_4^{2^{d-1}-2}x_5^{2^d-1} $& 73. & $x_1x_2^{2^{d-1}-1}x_3^{2^{d+1}-1}x_4^{2^d-1}x_5^{2^{d-1}-2} $\cr 
74. & $x_1x_2^{2^d-1}x_3^{2^{d-1}-2}x_4^{2^{d-1}-1}x_5^{2^{d+1}-1} $& 75. & $x_1x_2^{2^d-1}x_3^{2^{d-1}-2}x_4^{2^d-1}x_5^{2^{d+1}-2^{d-1}-1} $\cr 
76. & $x_1x_2^{2^d-1}x_3^{2^{d-1}-2}x_4^{2^{d+1}-2^{d-1}-1}x_5^{2^d-1} $& 77. & $x_1x_2^{2^d-1}x_3^{2^{d-1}-2}x_4^{2^{d+1}-1}x_5^{2^{d-1}-1} $\cr 
 \end{tabular}}
 \centerline{\begin{tabular}{llll}
78. & $x_1x_2^{2^d-1}x_3^{2^{d-1}-1}x_4^{2^{d-1}-2}x_5^{2^{d+1}-1} $& 79. & $x_1x_2^{2^d-1}x_3^{2^{d-1}-1}x_4^{2^d-1}x_5^{2^{d+1}-2^{d-1}-2} $\cr 
80. & $x_1x_2^{2^d-1}x_3^{2^{d-1}-1}x_4^{2^{d+1}-2^{d-1}-2}x_5^{2^d-1} $& 81. & $x_1x_2^{2^d-1}x_3^{2^{d-1}-1}x_4^{2^{d+1}-1}x_5^{2^{d-1}-2} $\cr 
82. & $x_1x_2^{2^d-1}x_3^{2^d-1}x_4^{2^{d-1}-2}x_5^{2^{d+1}-2^{d-1}-1} $& 83. & $x_1x_2^{2^d-1}x_3^{2^d-1}x_4^{2^{d-1}-1}x_5^{2^{d+1}-2^{d-1}-2} $\cr 
84. & $x_1x_2^{2^d-1}x_3^{2^d-1}x_4^{2^{d+1}-2^{d-1}-2}x_5^{2^{d-1}-1} $& 85. & $x_1x_2^{2^d-1}x_3^{2^d-1}x_4^{2^{d+1}-2^{d-1}-1}x_5^{2^{d-1}-2} $\cr 
86. & $x_1x_2^{2^d-1}x_3^{2^{d+1}-2^{d-1}-1}x_4^{2^{d-1}-2}x_5^{2^d-1} $& 87. & $x_1x_2^{2^d-1}x_3^{2^{d+1}-2^{d-1}-1}x_4^{2^d-1}x_5^{2^{d-1}-2} $\cr 
88. & $x_1x_2^{2^d-1}x_3^{2^{d+1}-1}x_4^{2^{d-1}-2}x_5^{2^{d-1}-1} $& 89. & $x_1x_2^{2^d-1}x_3^{2^{d+1}-1}x_4^{2^{d-1}-1}x_5^{2^{d-1}-2} $\cr 
90. & $x_1x_2^{2^{d+1}-1}x_3^{2^{d-1}-2}x_4^{2^{d-1}-1}x_5^{2^d-1} $& 91. & $x_1x_2^{2^{d+1}-1}x_3^{2^{d-1}-2}x_4^{2^d-1}x_5^{2^{d-1}-1} $\cr 
92. & $x_1x_2^{2^{d+1}-1}x_3^{2^{d-1}-1}x_4^{2^{d-1}-2}x_5^{2^d-1} $& 93. & $x_1x_2^{2^{d+1}-1}x_3^{2^{d-1}-1}x_4^{2^d-1}x_5^{2^{d-1}-2} $\cr 
94. & $x_1x_2^{2^{d+1}-1}x_3^{2^d-1}x_4^{2^{d-1}-2}x_5^{2^{d-1}-1} $& 95. & $x_1x_2^{2^{d+1}-1}x_3^{2^d-1}x_4^{2^{d-1}-1}x_5^{2^{d-1}-2} $\cr 
96. & $x_1^{3}x_2^{2^{d-1}-1}x_3^{2^{d-1}-1}x_4^{2^d-3}x_5^{2^{d+1}-2} $& 97. & $x_1^{3}x_2^{2^{d-1}-1}x_3^{2^{d-1}-1}x_4^{2^{d+1}-3}x_5^{2^d-2} $\cr 
98. & $x_1^{3}x_2^{2^{d-1}-1}x_3^{2^d-3}x_4^{2^{d-1}-2}x_5^{2^{d+1}-1} $& 99. & $x_1^{3}x_2^{2^{d-1}-1}x_3^{2^d-3}x_4^{2^{d-1}-1}x_5^{2^{d+1}-2} $\cr 
100. & $x_1^{3}x_2^{2^{d-1}-1}x_3^{2^d-3}x_4^{2^d-2}x_5^{2^{d+1}-2^{d-1}-1} $& 101. & $x_1^{3}x_2^{2^{d-1}-1}x_3^{2^d-3}x_4^{2^d-1}x_5^{2^{d+1}-2^{d-1}-2} $\cr 
102. & $x_1^{3}x_2^{2^{d-1}-1}x_3^{2^d-3}x_4^{2^{d+1}-2^{d-1}-2}x_5^{2^d-1} $& 103. & $x_1^{3}x_2^{2^{d-1}-1}x_3^{2^d-3}x_4^{2^{d+1}-2^{d-1}-1}x_5^{2^d-2} $\cr 
104. & $x_1^{3}x_2^{2^{d-1}-1}x_3^{2^d-3}x_4^{2^{d+1}-2}x_5^{2^{d-1}-1} $& 105. & $x_1^{3}x_2^{2^{d-1}-1}x_3^{2^d-3}x_4^{2^{d+1}-1}x_5^{2^{d-1}-2} $\cr 
106. & $x_1^{3}x_2^{2^{d-1}-1}x_3^{2^d-1}x_4^{2^d-3}x_5^{2^{d+1}-2^{d-1}-2} $& 107. & $x_1^{3}x_2^{2^{d-1}-1}x_3^{2^d-1}x_4^{2^{d+1}-2^{d-1}-3}x_5^{2^d-2} $\cr 
108. & $x_1^{3}x_2^{2^{d-1}-1}x_3^{2^d-1}x_4^{2^{d+1}-3}x_5^{2^{d-1}-2} $& 109. & $x_1^{3}x_2^{2^{d-1}-1}x_3^{2^{d+1}-3}x_4^{2^{d-1}-2}x_5^{2^d-1} $\cr 
110. & $x_1^{3}x_2^{2^{d-1}-1}x_3^{2^{d+1}-3}x_4^{2^{d-1}-1}x_5^{2^d-2} $& 111. & $x_1^{3}x_2^{2^{d-1}-1}x_3^{2^{d+1}-3}x_4^{2^d-2}x_5^{2^{d-1}-1} $\cr 
112. & $x_1^{3}x_2^{2^{d-1}-1}x_3^{2^{d+1}-3}x_4^{2^d-1}x_5^{2^{d-1}-2} $& 113. & $x_1^{3}x_2^{2^{d-1}-1}x_3^{2^{d+1}-1}x_4^{2^d-3}x_5^{2^{d-1}-2} $\cr 
114. & $x_1^{3}x_2^{2^d-3}x_3^{2^{d-1}-2}x_4^{2^{d-1}-1}x_5^{2^{d+1}-1} $& 115. & $x_1^{3}x_2^{2^d-3}x_3^{2^{d-1}-2}x_4^{2^d-1}x_5^{2^{d+1}-2^{d-1}-1} $\cr 
116. & $x_1^{3}x_2^{2^d-3}x_3^{2^{d-1}-2}x_4^{2^{d+1}-2^{d-1}-1}x_5^{2^d-1} $& 117. & $x_1^{3}x_2^{2^d-3}x_3^{2^{d-1}-2}x_4^{2^{d+1}-1}x_5^{2^{d-1}-1} $\cr 
118. & $x_1^{3}x_2^{2^d-3}x_3^{2^{d-1}-1}x_4^{2^{d-1}-2}x_5^{2^{d+1}-1} $& 119. & $x_1^{3}x_2^{2^d-3}x_3^{2^{d-1}-1}x_4^{2^{d-1}-1}x_5^{2^{d+1}-2} $\cr 
120. & $x_1^{3}x_2^{2^d-3}x_3^{2^{d-1}-1}x_4^{2^d-2}x_5^{2^{d+1}-2^{d-1}-1} $& 121. & $x_1^{3}x_2^{2^d-3}x_3^{2^{d-1}-1}x_4^{2^d-1}x_5^{2^{d+1}-2^{d-1}-2} $\cr 
122. & $x_1^{3}x_2^{2^d-3}x_3^{2^{d-1}-1}x_4^{2^{d+1}-2^{d-1}-2}x_5^{2^d-1} $& 123. & $x_1^{3}x_2^{2^d-3}x_3^{2^{d-1}-1}x_4^{2^{d+1}-2^{d-1}-1}x_5^{2^d-2} $\cr 
124. & $x_1^{3}x_2^{2^d-3}x_3^{2^{d-1}-1}x_4^{2^{d+1}-2}x_5^{2^{d-1}-1} $& 125. & $x_1^{3}x_2^{2^d-3}x_3^{2^{d-1}-1}x_4^{2^{d+1}-1}x_5^{2^{d-1}-2} $\cr 
126. & $x_1^{3}x_2^{2^d-3}x_3^{2^d-2}x_4^{2^{d-1}-1}x_5^{2^{d+1}-2^{d-1}-1} $& 127. & $x_1^{3}x_2^{2^d-3}x_3^{2^d-2}x_4^{2^{d+1}-2^{d-1}-1}x_5^{2^{d-1}-1} $\cr 
128. & $x_1^{3}x_2^{2^d-3}x_3^{2^d-1}x_4^{2^{d-1}-2}x_5^{2^{d+1}-2^{d-1}-1} $& 129. & $x_1^{3}x_2^{2^d-3}x_3^{2^d-1}x_4^{2^{d-1}-1}x_5^{2^{d+1}-2^{d-1}-2} $\cr 
130. & $x_1^{3}x_2^{2^d-3}x_3^{2^d-1}x_4^{2^{d+1}-2^{d-1}-2}x_5^{2^{d-1}-1} $& 131. & $x_1^{3}x_2^{2^d-3}x_3^{2^d-1}x_4^{2^{d+1}-2^{d-1}-1}x_5^{2^{d-1}-2} $\cr 
132. & $x_1^{3}x_2^{2^d-3}x_3^{2^{d+1}-2^{d-1}-2}x_4^{2^{d-1}-1}x_5^{2^d-1} $& 133. & $x_1^{3}x_2^{2^d-3}x_3^{2^{d+1}-2^{d-1}-2}x_4^{2^d-1}x_5^{2^{d-1}-1} $\cr 
134. & $x_1^{3}x_2^{2^d-3}x_3^{2^{d+1}-2^{d-1}-1}x_4^{2^{d-1}-2}x_5^{2^d-1} $& 135. & $x_1^{3}x_2^{2^d-3}x_3^{2^{d+1}-2^{d-1}-1}x_4^{2^{d-1}-1}x_5^{2^d-2} $\cr 
136. & $x_1^{3}x_2^{2^d-3}x_3^{2^{d+1}-2^{d-1}-1}x_4^{2^d-2}x_5^{2^{d-1}-1} $& 137. & $x_1^{3}x_2^{2^d-3}x_3^{2^{d+1}-2^{d-1}-1}x_4^{2^d-1}x_5^{2^{d-1}-2} $\cr 
138. & $x_1^{3}x_2^{2^d-3}x_3^{2^{d+1}-2}x_4^{2^{d-1}-1}x_5^{2^{d-1}-1} $& 139. & $x_1^{3}x_2^{2^d-3}x_3^{2^{d+1}-1}x_4^{2^{d-1}-2}x_5^{2^{d-1}-1} $\cr 
140. & $x_1^{3}x_2^{2^d-3}x_3^{2^{d+1}-1}x_4^{2^{d-1}-1}x_5^{2^{d-1}-2} $& 141. & $x_1^{3}x_2^{2^d-1}x_3^{2^{d-1}-1}x_4^{2^d-3}x_5^{2^{d+1}-2^{d-1}-2} $\cr 
142. & $x_1^{3}x_2^{2^d-1}x_3^{2^{d-1}-1}x_4^{2^{d+1}-2^{d-1}-3}x_5^{2^d-2} $& 143. & $x_1^{3}x_2^{2^d-1}x_3^{2^{d-1}-1}x_4^{2^{d+1}-3}x_5^{2^{d-1}-2} $\cr 
144. & $x_1^{3}x_2^{2^d-1}x_3^{2^d-3}x_4^{2^{d-1}-2}x_5^{2^{d+1}-2^{d-1}-1} $& 145. & $x_1^{3}x_2^{2^d-1}x_3^{2^d-3}x_4^{2^{d-1}-1}x_5^{2^{d+1}-2^{d-1}-2} $\cr 
146. & $x_1^{3}x_2^{2^d-1}x_3^{2^d-3}x_4^{2^{d+1}-2^{d-1}-2}x_5^{2^{d-1}-1} $& 147. & $x_1^{3}x_2^{2^d-1}x_3^{2^d-3}x_4^{2^{d+1}-2^{d-1}-1}x_5^{2^{d-1}-2} $\cr 
148. & $x_1^{3}x_2^{2^d-1}x_3^{2^d-1}x_4^{2^{d-1}-3}x_5^{2^{d+1}-2^{d-1}-2} $& 149. & $x_1^{3}x_2^{2^d-1}x_3^{2^d-1}x_4^{2^{d+1}-2^{d-1}-3}x_5^{2^{d-1}-2} $\cr 
150. & $x_1^{3}x_2^{2^d-1}x_3^{2^{d+1}-2^{d-1}-3}x_4^{2^{d-1}-2}x_5^{2^d-1} $& 151. & $x_1^{3}x_2^{2^d-1}x_3^{2^{d+1}-2^{d-1}-3}x_4^{2^{d-1}-1}x_5^{2^d-2} $\cr 
152. & $x_1^{3}x_2^{2^d-1}x_3^{2^{d+1}-2^{d-1}-3}x_4^{2^d-2}x_5^{2^{d-1}-1} $& 153. & $x_1^{3}x_2^{2^d-1}x_3^{2^{d+1}-2^{d-1}-3}x_4^{2^d-1}x_5^{2^{d-1}-2} $\cr 
154. & $x_1^{3}x_2^{2^d-1}x_3^{2^{d+1}-2^{d-1}-1}x_4^{2^{d-1}-3}x_5^{2^d-2} $& 155. & $x_1^{3}x_2^{2^d-1}x_3^{2^{d+1}-2^{d-1}-1}x_4^{2^d-3}x_5^{2^{d-1}-2} $\cr 
156. & $x_1^{3}x_2^{2^d-1}x_3^{2^{d+1}-3}x_4^{2^{d-1}-2}x_5^{2^{d-1}-1} $& 157. & $x_1^{3}x_2^{2^d-1}x_3^{2^{d+1}-3}x_4^{2^{d-1}-1}x_5^{2^{d-1}-2} $\cr 
 \end{tabular}}
 \centerline{\begin{tabular}{llll}
158. & $x_1^{3}x_2^{2^{d+1}-3}x_3^{2^{d-1}-2}x_4^{2^{d-1}-1}x_5^{2^d-1} $& 159. & $x_1^{3}x_2^{2^{d+1}-3}x_3^{2^{d-1}-2}x_4^{2^d-1}x_5^{2^{d-1}-1} $\cr 
160. & $x_1^{3}x_2^{2^{d+1}-3}x_3^{2^{d-1}-1}x_4^{2^{d-1}-2}x_5^{2^d-1} $& 161. & $x_1^{3}x_2^{2^{d+1}-3}x_3^{2^{d-1}-1}x_4^{2^{d-1}-1}x_5^{2^d-2} $\cr 
162. & $x_1^{3}x_2^{2^{d+1}-3}x_3^{2^{d-1}-1}x_4^{2^d-2}x_5^{2^{d-1}-1} $& 163. & $x_1^{3}x_2^{2^{d+1}-3}x_3^{2^{d-1}-1}x_4^{2^d-1}x_5^{2^{d-1}-2} $\cr 
164. & $x_1^{3}x_2^{2^{d+1}-3}x_3^{2^d-1}x_4^{2^{d-1}-2}x_5^{2^{d-1}-1} $& 165. & $x_1^{3}x_2^{2^{d+1}-3}x_3^{2^d-1}x_4^{2^{d-1}-1}x_5^{2^{d-1}-2} $\cr 
166. & $x_1^{3}x_2^{2^{d+1}-1}x_3^{2^{d-1}-1}x_4^{2^d-3}x_5^{2^{d-1}-2} $& 167. & $x_1^{3}x_2^{2^{d+1}-1}x_3^{2^d-3}x_4^{2^{d-1}-2}x_5^{2^{d-1}-1} $\cr 
168. & $x_1^{3}x_2^{2^{d+1}-1}x_3^{2^d-3}x_4^{2^{d-1}-1}x_5^{2^{d-1}-2} $& 169. & $x_1^{7}x_2^{2^{d+1}-5}x_3^{2^d-3}x_4^{2^{d-1}-2}x_5^{2^{d-1}-1} $\cr 
170. & $x_1^{7}x_2^{2^{d+1}-5}x_3^{2^d-3}x_4^{2^{d-1}-1}x_5^{2^{d-1}-2} $& 171. & $x_1^{2^{d-1}-1}x_2x_3^{2^{d-1}-2}x_4^{2^d-1}x_5^{2^{d+1}-1} $\cr 
172. & $x_1^{2^{d-1}-1}x_2x_3^{2^{d-1}-2}x_4^{2^{d+1}-1}x_5^{2^d-1} $& 173. & $x_1^{2^{d-1}-1}x_2x_3^{2^{d-1}-1}x_4^{2^d-2}x_5^{2^{d+1}-1} $\cr 
174. & $x_1^{2^{d-1}-1}x_2x_3^{2^{d-1}-1}x_4^{2^d-1}x_5^{2^{d+1}-2} $& 175. & $x_1^{2^{d-1}-1}x_2x_3^{2^{d-1}-1}x_4^{2^{d+1}-2}x_5^{2^d-1} $\cr 
176. & $x_1^{2^{d-1}-1}x_2x_3^{2^{d-1}-1}x_4^{2^{d+1}-1}x_5^{2^d-2} $& 177. & $x_1^{2^{d-1}-1}x_2x_3^{2^d-2}x_4^{2^{d-1}-1}x_5^{2^{d+1}-1} $\cr 
178. & $x_1^{2^{d-1}-1}x_2x_3^{2^d-2}x_4^{2^d-1}x_5^{2^{d+1}-2^{d-1}-1} $& 179. & $x_1^{2^{d-1}-1}x_2x_3^{2^d-2}x_4^{2^{d+1}-2^{d-1}-1}x_5^{2^d-1} $\cr 
180. & $x_1^{2^{d-1}-1}x_2x_3^{2^d-2}x_4^{2^{d+1}-1}x_5^{2^{d-1}-1} $& 181. & $x_1^{2^{d-1}-1}x_2x_3^{2^d-1}x_4^{2^{d-1}-2}x_5^{2^{d+1}-1} $\cr 
182. & $x_1^{2^{d-1}-1}x_2x_3^{2^d-1}x_4^{2^{d-1}-1}x_5^{2^{d+1}-2} $& 183. & $x_1^{2^{d-1}-1}x_2x_3^{2^d-1}x_4^{2^d-2}x_5^{2^{d+1}-2^{d-1}-1} $\cr 
184. & $x_1^{2^{d-1}-1}x_2x_3^{2^d-1}x_4^{2^d-1}x_5^{2^{d+1}-2^{d-1}-2} $& 185. & $x_1^{2^{d-1}-1}x_2x_3^{2^d-1}x_4^{2^{d+1}-2^{d-1}-2}x_5^{2^d-1} $\cr 
186. & $x_1^{2^{d-1}-1}x_2x_3^{2^d-1}x_4^{2^{d+1}-2^{d-1}-1}x_5^{2^d-2} $& 187. & $x_1^{2^{d-1}-1}x_2x_3^{2^d-1}x_4^{2^{d+1}-2}x_5^{2^{d-1}-1} $\cr 
188. & $x_1^{2^{d-1}-1}x_2x_3^{2^d-1}x_4^{2^{d+1}-1}x_5^{2^{d-1}-2} $& 189. & $x_1^{2^{d-1}-1}x_2x_3^{2^{d+1}-2}x_4^{2^{d-1}-1}x_5^{2^d-1} $\cr 
190. & $x_1^{2^{d-1}-1}x_2x_3^{2^{d+1}-2}x_4^{2^d-1}x_5^{2^{d-1}-1} $& 191. & $x_1^{2^{d-1}-1}x_2x_3^{2^{d+1}-1}x_4^{2^{d-1}-2}x_5^{2^d-1} $\cr 
192. & $x_1^{2^{d-1}-1}x_2x_3^{2^{d+1}-1}x_4^{2^{d-1}-1}x_5^{2^d-2} $& 193. & $x_1^{2^{d-1}-1}x_2x_3^{2^{d+1}-1}x_4^{2^d-2}x_5^{2^{d-1}-1} $\cr 
194. & $x_1^{2^{d-1}-1}x_2x_3^{2^{d+1}-1}x_4^{2^d-1}x_5^{2^{d-1}-2} $& 195. & $x_1^{2^{d-1}-1}x_2^{2^{d-1}-1}x_3x_4^{2^d-2}x_5^{2^{d+1}-1} $\cr 
196. & $x_1^{2^{d-1}-1}x_2^{2^{d-1}-1}x_3x_4^{2^d-1}x_5^{2^{d+1}-2} $& 197. & $x_1^{2^{d-1}-1}x_2^{2^{d-1}-1}x_3x_4^{2^{d+1}-2}x_5^{2^d-1} $\cr 
198. & $x_1^{2^{d-1}-1}x_2^{2^{d-1}-1}x_3x_4^{2^{d+1}-1}x_5^{2^d-2} $& 199. & $x_1^{2^{d-1}-1}x_2^{2^{d-1}-1}x_3^{2^d-1}x_4x_5^{2^{d+1}-2} $\cr 
200. & $x_1^{2^{d-1}-1}x_2^{2^{d-1}-1}x_3^{2^{d+1}-1}x_4x_5^{2^d-2} $& 201. & $x_1^{2^{d-1}-1}x_2^{2^d-1}x_3x_4^{2^{d-1}-2}x_5^{2^{d+1}-1} $\cr 
202. & $x_1^{2^{d-1}-1}x_2^{2^d-1}x_3x_4^{2^{d-1}-1}x_5^{2^{d+1}-2} $& 203. & $x_1^{2^{d-1}-1}x_2^{2^d-1}x_3x_4^{2^d-2}x_5^{2^{d+1}-2^{d-1}-1} $\cr 
204. & $x_1^{2^{d-1}-1}x_2^{2^d-1}x_3x_4^{2^d-1}x_5^{2^{d+1}-2^{d-1}-2} $& 205. & $x_1^{2^{d-1}-1}x_2^{2^d-1}x_3x_4^{2^{d+1}-2^{d-1}-2}x_5^{2^d-1} $\cr 
206. & $x_1^{2^{d-1}-1}x_2^{2^d-1}x_3x_4^{2^{d+1}-2^{d-1}-1}x_5^{2^d-2} $& 207. & $x_1^{2^{d-1}-1}x_2^{2^d-1}x_3x_4^{2^{d+1}-2}x_5^{2^{d-1}-1} $\cr 
208. & $x_1^{2^{d-1}-1}x_2^{2^d-1}x_3x_4^{2^{d+1}-1}x_5^{2^{d-1}-2} $& 209. & $x_1^{2^{d-1}-1}x_2^{2^d-1}x_3^{2^{d-1}-1}x_4x_5^{2^{d+1}-2} $\cr 
210. & $x_1^{2^{d-1}-1}x_2^{2^d-1}x_3^{2^d-1}x_4x_5^{2^{d+1}-2^{d-1}-2} $& 211. & $x_1^{2^{d-1}-1}x_2^{2^d-1}x_3^{2^{d+1}-2^{d-1}-1}x_4x_5^{2^d-2} $\cr 
212. & $x_1^{2^{d-1}-1}x_2^{2^d-1}x_3^{2^{d+1}-1}x_4x_5^{2^{d-1}-2} $& 213. & $x_1^{2^{d-1}-1}x_2^{2^{d+1}-1}x_3x_4^{2^{d-1}-2}x_5^{2^d-1} $\cr 
214. & $x_1^{2^{d-1}-1}x_2^{2^{d+1}-1}x_3x_4^{2^{d-1}-1}x_5^{2^d-2} $& 215. & $x_1^{2^{d-1}-1}x_2^{2^{d+1}-1}x_3x_4^{2^d-2}x_5^{2^{d-1}-1} $\cr 
216. & $x_1^{2^{d-1}-1}x_2^{2^{d+1}-1}x_3x_4^{2^d-1}x_5^{2^{d-1}-2} $& 217. & $x_1^{2^{d-1}-1}x_2^{2^{d+1}-1}x_3^{2^{d-1}-1}x_4x_5^{2^d-2} $\cr 
218. & $x_1^{2^{d-1}-1}x_2^{2^{d+1}-1}x_3^{2^d-1}x_4x_5^{2^{d-1}-2} $& 219. & $x_1^{2^d-1}x_2x_3^{2^{d-1}-2}x_4^{2^{d-1}-1}x_5^{2^{d+1}-1} $\cr 
220. & $x_1^{2^d-1}x_2x_3^{2^{d-1}-2}x_4^{2^d-1}x_5^{2^{d+1}-2^{d-1}-1} $& 221. & $x_1^{2^d-1}x_2x_3^{2^{d-1}-2}x_4^{2^{d+1}-2^{d-1}-1}x_5^{2^d-1} $\cr 
222. & $x_1^{2^d-1}x_2x_3^{2^{d-1}-2}x_4^{2^{d+1}-1}x_5^{2^{d-1}-1} $& 223. & $x_1^{2^d-1}x_2x_3^{2^{d-1}-1}x_4^{2^{d-1}-2}x_5^{2^{d+1}-1} $\cr 
224. & $x_1^{2^d-1}x_2x_3^{2^{d-1}-1}x_4^{2^d-1}x_5^{2^{d+1}-2^{d-1}-2} $& 225. & $x_1^{2^d-1}x_2x_3^{2^{d-1}-1}x_4^{2^{d+1}-2^{d-1}-2}x_5^{2^d-1} $\cr 
226. & $x_1^{2^d-1}x_2x_3^{2^{d-1}-1}x_4^{2^{d+1}-1}x_5^{2^{d-1}-2} $& 227. & $x_1^{2^d-1}x_2x_3^{2^d-1}x_4^{2^{d-1}-2}x_5^{2^{d+1}-2^{d-1}-1} $\cr 
228. & $x_1^{2^d-1}x_2x_3^{2^d-1}x_4^{2^{d-1}-1}x_5^{2^{d+1}-2^{d-1}-2} $& 229. & $x_1^{2^d-1}x_2x_3^{2^d-1}x_4^{2^{d+1}-2^{d-1}-2}x_5^{2^{d-1}-1} $\cr 
230. & $x_1^{2^d-1}x_2x_3^{2^d-1}x_4^{2^{d+1}-2^{d-1}-1}x_5^{2^{d-1}-2} $& 231. & $x_1^{2^d-1}x_2x_3^{2^{d+1}-2^{d-1}-2}x_4^{2^{d-1}-1}x_5^{2^d-1} $\cr 
232. & $x_1^{2^d-1}x_2x_3^{2^{d+1}-2^{d-1}-2}x_4^{2^d-1}x_5^{2^{d-1}-1} $& 233. & $x_1^{2^d-1}x_2x_3^{2^{d+1}-2^{d-1}-1}x_4^{2^{d-1}-2}x_5^{2^d-1} $\cr 
234. & $x_1^{2^d-1}x_2x_3^{2^{d+1}-2^{d-1}-1}x_4^{2^d-1}x_5^{2^{d-1}-2} $& 235. & $x_1^{2^d-1}x_2x_3^{2^{d+1}-1}x_4^{2^{d-1}-2}x_5^{2^{d-1}-1} $\cr 
236. & $x_1^{2^d-1}x_2x_3^{2^{d+1}-1}x_4^{2^{d-1}-1}x_5^{2^{d-1}-2} $& 237. & $x_1^{2^d-1}x_2^{3}x_3^{2^{d-1}-1}x_4^{2^d-3}x_5^{2^{d+1}-2^{d-1}-2} $\cr 
 \end{tabular}}
 \centerline{\begin{tabular}{llll}
238. & $x_1^{2^d-1}x_2^{3}x_3^{2^{d-1}-1}x_4^{2^{d+1}-2^{d-1}-3}x_5^{2^d-2} $& 239. & $x_1^{2^d-1}x_2^{3}x_3^{2^{d-1}-1}x_4^{2^{d+1}-3}x_5^{2^{d-1}-2} $\cr 
240. & $x_1^{2^d-1}x_2^{3}x_3^{2^d-3}x_4^{2^{d-1}-2}x_5^{2^{d+1}-2^{d-1}-1} $& 241. & $x_1^{2^d-1}x_2^{3}x_3^{2^d-3}x_4^{2^{d-1}-1}x_5^{2^{d+1}-2^{d-1}-2} $\cr 
242. & $x_1^{2^d-1}x_2^{3}x_3^{2^d-3}x_4^{2^{d+1}-2^{d-1}-2}x_5^{2^{d-1}-1} $& 243. & $x_1^{2^d-1}x_2^{3}x_3^{2^d-3}x_4^{2^{d+1}-2^{d-1}-1}x_5^{2^{d-1}-2} $\cr 
244. & $x_1^{2^d-1}x_2^{3}x_3^{2^d-1}x_4^{2^{d+1}-2^{d-1}-3}x_5^{2^{d-1}-2} $& 245. & $x_1^{2^d-1}x_2^{3}x_3^{2^{d+1}-2^{d-1}-3}x_4^{2^{d-1}-2}x_5^{2^d-1} $\cr 
246. & $x_1^{2^d-1}x_2^{3}x_3^{2^{d+1}-2^{d-1}-3}x_4^{2^{d-1}-1}x_5^{2^d-2} $& 247. & $x_1^{2^d-1}x_2^{3}x_3^{2^{d+1}-2^{d-1}-3}x_4^{2^d-2}x_5^{2^{d-1}-1} $\cr 
248. & $x_1^{2^d-1}x_2^{3}x_3^{2^{d+1}-2^{d-1}-3}x_4^{2^d-1}x_5^{2^{d-1}-2} $& 249. & $x_1^{2^d-1}x_2^{3}x_3^{2^{d+1}-2^{d-1}-1}x_4^{2^d-3}x_5^{2^{d-1}-2} $\cr 
250. & $x_1^{2^d-1}x_2^{3}x_3^{2^{d+1}-3}x_4^{2^{d-1}-2}x_5^{2^{d-1}-1} $& 251. & $x_1^{2^d-1}x_2^{3}x_3^{2^{d+1}-3}x_4^{2^{d-1}-1}x_5^{2^{d-1}-2} $\cr 
252. & $x_1^{2^d-1}x_2^{2^{d-1}-1}x_3x_4^{2^{d-1}-2}x_5^{2^{d+1}-1} $& 253. & $x_1^{2^d-1}x_2^{2^{d-1}-1}x_3x_4^{2^{d-1}-1}x_5^{2^{d+1}-2} $\cr 
254. & $x_1^{2^d-1}x_2^{2^{d-1}-1}x_3x_4^{2^d-2}x_5^{2^{d+1}-2^{d-1}-1} $& 255. & $x_1^{2^d-1}x_2^{2^{d-1}-1}x_3x_4^{2^d-1}x_5^{2^{d+1}-2^{d-1}-2} $\cr 
256. & $x_1^{2^d-1}x_2^{2^{d-1}-1}x_3x_4^{2^{d+1}-2^{d-1}-2}x_5^{2^d-1} $& 257. & $x_1^{2^d-1}x_2^{2^{d-1}-1}x_3x_4^{2^{d+1}-2^{d-1}-1}x_5^{2^d-2} $\cr 
258. & $x_1^{2^d-1}x_2^{2^{d-1}-1}x_3x_4^{2^{d+1}-2}x_5^{2^{d-1}-1} $& 259. & $x_1^{2^d-1}x_2^{2^{d-1}-1}x_3x_4^{2^{d+1}-1}x_5^{2^{d-1}-2} $\cr 
260. & $x_1^{2^d-1}x_2^{2^{d-1}-1}x_3^{2^{d-1}-1}x_4x_5^{2^{d+1}-2} $& 261. & $x_1^{2^d-1}x_2^{2^{d-1}-1}x_3^{2^d-1}x_4x_5^{2^{d+1}-2^{d-1}-2} $\cr 
262. & $x_1^{2^d-1}x_2^{2^{d-1}-1}x_3^{2^{d+1}-2^{d-1}-1}x_4x_5^{2^d-2} $& 263. & $x_1^{2^d-1}x_2^{2^{d-1}-1}x_3^{2^{d+1}-1}x_4x_5^{2^{d-1}-2} $\cr 
264. & $x_1^{2^d-1}x_2^{2^d-1}x_3x_4^{2^{d-1}-2}x_5^{2^{d+1}-2^{d-1}-1} $& 265. & $x_1^{2^d-1}x_2^{2^d-1}x_3x_4^{2^{d-1}-1}x_5^{2^{d+1}-2^{d-1}-2} $\cr 
266. & $x_1^{2^d-1}x_2^{2^d-1}x_3x_4^{2^{d+1}-2^{d-1}-2}x_5^{2^{d-1}-1} $& 267. & $x_1^{2^d-1}x_2^{2^d-1}x_3x_4^{2^{d+1}-2^{d-1}-1}x_5^{2^{d-1}-2} $\cr 
268. & $x_1^{2^d-1}x_2^{2^d-1}x_3^{3}x_4^{2^{d+1}-2^{d-1}-3}x_5^{2^{d-1}-2} $& 269. & $x_1^{2^d-1}x_2^{2^d-1}x_3^{2^{d-1}-1}x_4x_5^{2^{d+1}-2^{d-1}-2} $\cr 
270. & $x_1^{2^d-1}x_2^{2^d-1}x_3^{2^{d+1}-2^{d-1}-1}x_4x_5^{2^{d-1}-2} $& 271. & $x_1^{2^d-1}x_2^{2^{d+1}-2^{d-1}-1}x_3x_4^{2^{d-1}-2}x_5^{2^d-1} $\cr 
272. & $x_1^{2^d-1}x_2^{2^{d+1}-2^{d-1}-1}x_3x_4^{2^d-1}x_5^{2^{d-1}-2} $& 273. & $x_1^{2^d-1}x_2^{2^{d+1}-2^{d-1}-1}x_3^{3}x_4^{2^d-3}x_5^{2^{d-1}-2} $\cr 
274. & $x_1^{2^d-1}x_2^{2^{d+1}-2^{d-1}-1}x_3^{2^{d-1}-1}x_4x_5^{2^d-2} $& 275. & $x_1^{2^d-1}x_2^{2^{d+1}-2^{d-1}-1}x_3^{2^d-1}x_4x_5^{2^{d-1}-2} $\cr 
276. & $x_1^{2^d-1}x_2^{2^{d+1}-1}x_3x_4^{2^{d-1}-2}x_5^{2^{d-1}-1} $& 277. & $x_1^{2^d-1}x_2^{2^{d+1}-1}x_3x_4^{2^{d-1}-1}x_5^{2^{d-1}-2} $\cr 
278. & $x_1^{2^d-1}x_2^{2^{d+1}-1}x_3^{2^{d-1}-1}x_4x_5^{2^{d-1}-2} $& 279. & $x_1^{2^{d+1}-1}x_2x_3^{2^{d-1}-2}x_4^{2^{d-1}-1}x_5^{2^d-1} $\cr 
280. & $x_1^{2^{d+1}-1}x_2x_3^{2^{d-1}-2}x_4^{2^d-1}x_5^{2^{d-1}-1} $& 281. & $x_1^{2^{d+1}-1}x_2x_3^{2^{d-1}-1}x_4^{2^{d-1}-2}x_5^{2^d-1} $\cr 
282. & $x_1^{2^{d+1}-1}x_2x_3^{2^{d-1}-1}x_4^{2^d-1}x_5^{2^{d-1}-2} $& 283. & $x_1^{2^{d+1}-1}x_2x_3^{2^d-1}x_4^{2^{d-1}-2}x_5^{2^{d-1}-1} $\cr 
284. & $x_1^{2^{d+1}-1}x_2x_3^{2^d-1}x_4^{2^{d-1}-1}x_5^{2^{d-1}-2} $& 285. & $x_1^{2^{d+1}-1}x_2^{3}x_3^{2^{d-1}-1}x_4^{2^d-3}x_5^{2^{d-1}-2} $\cr 
286. & $x_1^{2^{d+1}-1}x_2^{3}x_3^{2^d-3}x_4^{2^{d-1}-2}x_5^{2^{d-1}-1} $& 287. & $x_1^{2^{d+1}-1}x_2^{3}x_3^{2^d-3}x_4^{2^{d-1}-1}x_5^{2^{d-1}-2} $\cr 
288. & $x_1^{2^{d+1}-1}x_2^{2^{d-1}-1}x_3x_4^{2^{d-1}-2}x_5^{2^d-1} $& 289. & $x_1^{2^{d+1}-1}x_2^{2^{d-1}-1}x_3x_4^{2^{d-1}-1}x_5^{2^d-2} $\cr 
290. & $x_1^{2^{d+1}-1}x_2^{2^{d-1}-1}x_3x_4^{2^d-2}x_5^{2^{d-1}-1} $& 291. & $x_1^{2^{d+1}-1}x_2^{2^{d-1}-1}x_3x_4^{2^d-1}x_5^{2^{d-1}-2} $\cr 
292. & $x_1^{2^{d+1}-1}x_2^{2^{d-1}-1}x_3^{2^{d-1}-1}x_4x_5^{2^d-2} $& 293. & $x_1^{2^{d+1}-1}x_2^{2^{d-1}-1}x_3^{2^d-1}x_4x_5^{2^{d-1}-2} $\cr 
294. & $x_1^{2^{d+1}-1}x_2^{2^d-1}x_3x_4^{2^{d-1}-2}x_5^{2^{d-1}-1} $& 295. & $x_1^{2^{d+1}-1}x_2^{2^d-1}x_3x_4^{2^{d-1}-1}x_5^{2^{d-1}-2} $\cr 
296. & $x_1^{2^{d+1}-1}x_2^{2^d-1}x_3^{2^{d-1}-1}x_4x_5^{2^{d-1}-2} $&
 \end{tabular}}

\medskip
For $d = 3$,

\medskip
 \centerline{\begin{tabular}{llll}
297. \ $x_1^{3}x_2^{3}x_3^{3}x_4^{4}x_5^{15} $& 298. \ $x_1^{3}x_2^{3}x_3^{3}x_4^{7}x_5^{12} $& 299. \ $x_1^{3}x_2^{3}x_3^{3}x_4^{12}x_5^{7} $& 300. \ $x_1^{3}x_2^{3}x_3^{3}x_4^{15}x_5^{4} $\cr  
301. \ $x_1^{3}x_2^{3}x_3^{4}x_4^{3}x_5^{15} $& 302. \ $x_1^{3}x_2^{3}x_3^{4}x_4^{7}x_5^{11} $& 303. \ $x_1^{3}x_2^{3}x_3^{4}x_4^{11}x_5^{7} $& 304. \ $x_1^{3}x_2^{3}x_3^{4}x_4^{15}x_5^{3} $\cr  
305. \ $x_1^{3}x_2^{3}x_3^{7}x_4^{3}x_5^{12} $& 306. \ $x_1^{3}x_2^{3}x_3^{7}x_4^{4}x_5^{11} $& 307. \ $x_1^{3}x_2^{3}x_3^{7}x_4^{7}x_5^{8} $& 308. \ $x_1^{3}x_2^{3}x_3^{7}x_4^{8}x_5^{7} $\cr  
309. \ $x_1^{3}x_2^{3}x_3^{7}x_4^{11}x_5^{4} $& 310. \ $x_1^{3}x_2^{3}x_3^{7}x_4^{12}x_5^{3} $& 311. \ $x_1^{3}x_2^{3}x_3^{12}x_4^{3}x_5^{7} $& 312. \ $x_1^{3}x_2^{3}x_3^{12}x_4^{7}x_5^{3} $\cr  
313. \ $x_1^{3}x_2^{3}x_3^{15}x_4^{3}x_5^{4} $& 314. \ $x_1^{3}x_2^{3}x_3^{15}x_4^{4}x_5^{3} $& 315. \ $x_1^{3}x_2^{7}x_3^{3}x_4^{3}x_5^{12} $& 316. \ $x_1^{3}x_2^{7}x_3^{3}x_4^{4}x_5^{11} $\cr  
317. \ $x_1^{3}x_2^{7}x_3^{3}x_4^{7}x_5^{8} $& 318. \ $x_1^{3}x_2^{7}x_3^{3}x_4^{8}x_5^{7} $& 319. \ $x_1^{3}x_2^{7}x_3^{3}x_4^{11}x_5^{4} $& 320. \ $x_1^{3}x_2^{7}x_3^{3}x_4^{12}x_5^{3} $\cr  
321. \ $x_1^{3}x_2^{7}x_3^{7}x_4^{8}x_5^{3} $& 322. \ $x_1^{3}x_2^{7}x_3^{8}x_4^{3}x_5^{7} $& 323. \ $x_1^{3}x_2^{7}x_3^{8}x_4^{7}x_5^{3} $& 324. \ $x_1^{3}x_2^{7}x_3^{11}x_4^{4}x_5^{3} $\cr  
325. \ $x_1^{3}x_2^{15}x_3^{3}x_4^{3}x_5^{4} $& 326. \ $x_1^{3}x_2^{15}x_3^{3}x_4^{4}x_5^{3} $& 327. \ $x_1^{7}x_2^{3}x_3^{3}x_4^{3}x_5^{12} $& 328. \ $x_1^{7}x_2^{3}x_3^{3}x_4^{4}x_5^{11} $\cr  
329. \ $x_1^{7}x_2^{3}x_3^{3}x_4^{7}x_5^{8} $& 330. \ $x_1^{7}x_2^{3}x_3^{3}x_4^{8}x_5^{7} $& 331. \ $x_1^{7}x_2^{3}x_3^{3}x_4^{11}x_5^{4} $& 332. \ $x_1^{7}x_2^{3}x_3^{3}x_4^{12}x_5^{3} $\cr  
333. \ $x_1^{7}x_2^{3}x_3^{7}x_4^{3}x_5^{8} $& 334. \ $x_1^{7}x_2^{3}x_3^{7}x_4^{8}x_5^{3} $& 335. \ $x_1^{7}x_2^{3}x_3^{8}x_4^{3}x_5^{7} $& 336. \ $x_1^{7}x_2^{3}x_3^{8}x_4^{7}x_5^{3} $\cr  
337. \ $x_1^{7}x_2^{3}x_3^{11}x_4^{3}x_5^{4} $& 338. \ $x_1^{7}x_2^{3}x_3^{11}x_4^{4}x_5^{3} $& 339. \ $x_1^{7}x_2^{7}x_3^{3}x_4^{3}x_5^{8} $& 340. \ $x_1^{7}x_2^{7}x_3^{3}x_4^{8}x_5^{3} $\cr  
 \end{tabular}}
 \centerline{\begin{tabular}{llll}
341. \ $x_1^{7}x_2^{7}x_3^{8}x_4^{3}x_5^{3} $& 342. \ $x_1^{7}x_2^{7}x_3^{9}x_4^{2}x_5^{3} $& 343. \ $x_1^{7}x_2^{7}x_3^{9}x_4^{3}x_5^{2} $& 344. \ $x_1^{7}x_2^{9}x_3^{2}x_4^{3}x_5^{7} $\cr  
345. \ $x_1^{7}x_2^{9}x_3^{2}x_4^{7}x_5^{3} $& 346. \ $x_1^{7}x_2^{9}x_3^{3}x_4^{2}x_5^{7} $& 347. \ $x_1^{7}x_2^{9}x_3^{3}x_4^{3}x_5^{6} $& 348. \ $x_1^{7}x_2^{9}x_3^{3}x_4^{6}x_5^{3} $\cr  
349. \ $x_1^{7}x_2^{9}x_3^{3}x_4^{7}x_5^{2} $& 350. \ $x_1^{7}x_2^{9}x_3^{7}x_4^{2}x_5^{3} $& 351. \ $x_1^{7}x_2^{9}x_3^{7}x_4^{3}x_5^{2} $& 352. \ $x_1^{7}x_2^{11}x_3^{3}x_4^{3}x_5^{4} $\cr  
353. \ $x_1^{7}x_2^{11}x_3^{3}x_4^{4}x_5^{3} $& 354. \ $x_1^{15}x_2^{3}x_3^{3}x_4^{3}x_5^{4} $& 355. \ $x_1^{15}x_2^{3}x_3^{3}x_4^{4}x_5^{3} $&\cr
 \end{tabular}}

\medskip
For $d \geqslant 4$,

\medskip
 \centerline{\begin{tabular}{llll}
297. & $x_1^{3}x_2^{2^{d-1}-3}x_3^{2^{d-1}-2}x_4^{2^d-1}x_5^{2^{d+1}-1} $& 298. & $x_1^{3}x_2^{2^{d-1}-3}x_3^{2^{d-1}-2}x_4^{2^{d+1}-1}x_5^{2^d-1} $\cr 
299. & $x_1^{3}x_2^{2^{d-1}-3}x_3^{2^{d-1}-1}x_4^{2^d-2}x_5^{2^{d+1}-1} $& 300. & $x_1^{3}x_2^{2^{d-1}-3}x_3^{2^{d-1}-1}x_4^{2^d-1}x_5^{2^{d+1}-2} $\cr 
301. & $x_1^{3}x_2^{2^{d-1}-3}x_3^{2^{d-1}-1}x_4^{2^{d+1}-2}x_5^{2^d-1} $& 302. & $x_1^{3}x_2^{2^{d-1}-3}x_3^{2^{d-1}-1}x_4^{2^{d+1}-1}x_5^{2^d-2} $\cr 
303. & $x_1^{3}x_2^{2^{d-1}-3}x_3^{2^d-2}x_4^{2^{d-1}-1}x_5^{2^{d+1}-1} $& 304. & $x_1^{3}x_2^{2^{d-1}-3}x_3^{2^d-2}x_4^{2^d-1}x_5^{2^{d+1}-2^{d-1}-1} $\cr 
305. & $x_1^{3}x_2^{2^{d-1}-3}x_3^{2^d-2}x_4^{2^{d+1}-2^{d-1}-1}x_5^{2^d-1} $& 306. & $x_1^{3}x_2^{2^{d-1}-3}x_3^{2^d-2}x_4^{2^{d+1}-1}x_5^{2^{d-1}-1} $\cr 
307. & $x_1^{3}x_2^{2^{d-1}-3}x_3^{2^d-1}x_4^{2^{d-1}-2}x_5^{2^{d+1}-1} $& 308. & $x_1^{3}x_2^{2^{d-1}-3}x_3^{2^d-1}x_4^{2^{d-1}-1}x_5^{2^{d+1}-2} $\cr 
309. & $x_1^{3}x_2^{2^{d-1}-3}x_3^{2^d-1}x_4^{2^d-2}x_5^{2^{d+1}-2^{d-1}-1} $& 310. & $x_1^{3}x_2^{2^{d-1}-3}x_3^{2^d-1}x_4^{2^d-1}x_5^{2^{d+1}-2^{d-1}-2} $\cr 
311. & $x_1^{3}x_2^{2^{d-1}-3}x_3^{2^d-1}x_4^{2^{d+1}-2^{d-1}-2}x_5^{2^d-1} $& 312. & $x_1^{3}x_2^{2^{d-1}-3}x_3^{2^d-1}x_4^{2^{d+1}-2^{d-1}-1}x_5^{2^d-2} $\cr 
313. & $x_1^{3}x_2^{2^{d-1}-3}x_3^{2^d-1}x_4^{2^{d+1}-2}x_5^{2^{d-1}-1} $& 314. & $x_1^{3}x_2^{2^{d-1}-3}x_3^{2^d-1}x_4^{2^{d+1}-1}x_5^{2^{d-1}-2} $\cr 
315. & $x_1^{3}x_2^{2^{d-1}-3}x_3^{2^{d+1}-2}x_4^{2^{d-1}-1}x_5^{2^d-1} $& 316. & $x_1^{3}x_2^{2^{d-1}-3}x_3^{2^{d+1}-2}x_4^{2^d-1}x_5^{2^{d-1}-1} $\cr 
317. & $x_1^{3}x_2^{2^{d-1}-3}x_3^{2^{d+1}-1}x_4^{2^{d-1}-2}x_5^{2^d-1} $& 318. & $x_1^{3}x_2^{2^{d-1}-3}x_3^{2^{d+1}-1}x_4^{2^{d-1}-1}x_5^{2^d-2} $\cr 
319. & $x_1^{3}x_2^{2^{d-1}-3}x_3^{2^{d+1}-1}x_4^{2^d-2}x_5^{2^{d-1}-1} $& 320. & $x_1^{3}x_2^{2^{d-1}-3}x_3^{2^{d+1}-1}x_4^{2^d-1}x_5^{2^{d-1}-2} $\cr 
321. & $x_1^{3}x_2^{2^{d-1}-1}x_3^{2^{d-1}-3}x_4^{2^d-2}x_5^{2^{d+1}-1} $& 322. & $x_1^{3}x_2^{2^{d-1}-1}x_3^{2^{d-1}-3}x_4^{2^d-1}x_5^{2^{d+1}-2} $\cr 
323. & $x_1^{3}x_2^{2^{d-1}-1}x_3^{2^{d-1}-3}x_4^{2^{d+1}-2}x_5^{2^d-1} $& 324. & $x_1^{3}x_2^{2^{d-1}-1}x_3^{2^{d-1}-3}x_4^{2^{d+1}-1}x_5^{2^d-2} $\cr 
325. & $x_1^{3}x_2^{2^{d-1}-1}x_3^{2^d-1}x_4^{2^{d-1}-3}x_5^{2^{d+1}-2} $& 326. & $x_1^{3}x_2^{2^{d-1}-1}x_3^{2^{d+1}-1}x_4^{2^{d-1}-3}x_5^{2^d-2} $\cr 
327. & $x_1^{3}x_2^{2^d-1}x_3^{2^{d-1}-3}x_4^{2^{d-1}-2}x_5^{2^{d+1}-1} $& 328. & $x_1^{3}x_2^{2^d-1}x_3^{2^{d-1}-3}x_4^{2^{d-1}-1}x_5^{2^{d+1}-2} $\cr 
329. & $x_1^{3}x_2^{2^d-1}x_3^{2^{d-1}-3}x_4^{2^d-2}x_5^{2^{d+1}-2^{d-1}-1} $& 330. & $x_1^{3}x_2^{2^d-1}x_3^{2^{d-1}-3}x_4^{2^d-1}x_5^{2^{d+1}-2^{d-1}-2} $\cr 
331. & $x_1^{3}x_2^{2^d-1}x_3^{2^{d-1}-3}x_4^{2^{d+1}-2^{d-1}-2}x_5^{2^d-1} $& 332. & $x_1^{3}x_2^{2^d-1}x_3^{2^{d-1}-3}x_4^{2^{d+1}-2^{d-1}-1}x_5^{2^d-2} $\cr 
333. & $x_1^{3}x_2^{2^d-1}x_3^{2^{d-1}-3}x_4^{2^{d+1}-2}x_5^{2^{d-1}-1} $& 334. & $x_1^{3}x_2^{2^d-1}x_3^{2^{d-1}-3}x_4^{2^{d+1}-1}x_5^{2^{d-1}-2} $\cr 
335. & $x_1^{3}x_2^{2^d-1}x_3^{2^{d-1}-1}x_4^{2^{d-1}-3}x_5^{2^{d+1}-2} $& 336. & $x_1^{3}x_2^{2^d-1}x_3^{2^{d+1}-1}x_4^{2^{d-1}-3}x_5^{2^{d-1}-2} $\cr 
337. & $x_1^{3}x_2^{2^{d+1}-1}x_3^{2^{d-1}-3}x_4^{2^{d-1}-2}x_5^{2^d-1} $& 338. & $x_1^{3}x_2^{2^{d+1}-1}x_3^{2^{d-1}-3}x_4^{2^{d-1}-1}x_5^{2^d-2} $\cr 
339. & $x_1^{3}x_2^{2^{d+1}-1}x_3^{2^{d-1}-3}x_4^{2^d-2}x_5^{2^{d-1}-1} $& 340. & $x_1^{3}x_2^{2^{d+1}-1}x_3^{2^{d-1}-3}x_4^{2^d-1}x_5^{2^{d-1}-2} $\cr 
341. & $x_1^{3}x_2^{2^{d+1}-1}x_3^{2^{d-1}-1}x_4^{2^{d-1}-3}x_5^{2^d-2} $& 342. & $x_1^{3}x_2^{2^{d+1}-1}x_3^{2^d-1}x_4^{2^{d-1}-3}x_5^{2^{d-1}-2} $\cr 
343. & $x_1^{7}x_2^{2^{d-1}-1}x_3^{2^d-5}x_4^{2^{d-1}-3}x_5^{2^{d+1}-2} $& 344. & $x_1^{7}x_2^{2^{d-1}-1}x_3^{2^d-5}x_4^{2^d-3}x_5^{2^{d+1}-2^{d-1}-2} $\cr 
345. & $x_1^{7}x_2^{2^{d-1}-1}x_3^{2^d-5}x_4^{2^{d+1}-2^{d-1}-3}x_5^{2^d-2} $& 346. & $x_1^{7}x_2^{2^{d-1}-1}x_3^{2^d-5}x_4^{2^{d+1}-3}x_5^{2^{d-1}-2} $\cr 
347. & $x_1^{7}x_2^{2^{d-1}-1}x_3^{2^{d+1}-5}x_4^{2^{d-1}-3}x_5^{2^d-2} $& 348. & $x_1^{7}x_2^{2^{d-1}-1}x_3^{2^{d+1}-5}x_4^{2^d-3}x_5^{2^{d-1}-2} $\cr 
349. & $x_1^{7}x_2^{2^d-5}x_3^{2^{d-1}-3}x_4^{2^{d-1}-2}x_5^{2^{d+1}-1} $& 350. & $x_1^{7}x_2^{2^d-5}x_3^{2^{d-1}-3}x_4^{2^{d-1}-1}x_5^{2^{d+1}-2} $\cr 
351. & $x_1^{7}x_2^{2^d-5}x_3^{2^{d-1}-3}x_4^{2^d-2}x_5^{2^{d+1}-2^{d-1}-1} $& 352. & $x_1^{7}x_2^{2^d-5}x_3^{2^{d-1}-3}x_4^{2^d-1}x_5^{2^{d+1}-2^{d-1}-2} $\cr 
353. & $x_1^{7}x_2^{2^d-5}x_3^{2^{d-1}-3}x_4^{2^{d+1}-2^{d-1}-2}x_5^{2^d-1} $& 354. & $x_1^{7}x_2^{2^d-5}x_3^{2^{d-1}-3}x_4^{2^{d+1}-2^{d-1}-1}x_5^{2^d-2} $\cr 
355. & $x_1^{7}x_2^{2^d-5}x_3^{2^{d-1}-3}x_4^{2^{d+1}-2}x_5^{2^{d-1}-1} $& 356. & $x_1^{7}x_2^{2^d-5}x_3^{2^{d-1}-3}x_4^{2^{d+1}-1}x_5^{2^{d-1}-2} $\cr 
357. & $x_1^{7}x_2^{2^d-5}x_3^{2^{d-1}-1}x_4^{2^{d-1}-3}x_5^{2^{d+1}-2} $& 358. & $x_1^{7}x_2^{2^d-5}x_3^{2^{d-1}-1}x_4^{2^d-3}x_5^{2^{d+1}-2^{d-1}-2} $\cr 
359. & $x_1^{7}x_2^{2^d-5}x_3^{2^{d-1}-1}x_4^{2^{d+1}-2^{d-1}-3}x_5^{2^d-2} $& 360. & $x_1^{7}x_2^{2^d-5}x_3^{2^{d-1}-1}x_4^{2^{d+1}-3}x_5^{2^{d-1}-2} $\cr 
361. & $x_1^{7}x_2^{2^d-5}x_3^{2^d-3}x_4^{2^{d-1}-2}x_5^{2^{d+1}-2^{d-1}-1} $& 362. & $x_1^{7}x_2^{2^d-5}x_3^{2^d-3}x_4^{2^{d-1}-1}x_5^{2^{d+1}-2^{d-1}-2} $\cr 
363. & $x_1^{7}x_2^{2^d-5}x_3^{2^d-3}x_4^{2^{d+1}-2^{d-1}-2}x_5^{2^{d-1}-1} $& 364. & $x_1^{7}x_2^{2^d-5}x_3^{2^d-3}x_4^{2^{d+1}-2^{d-1}-1}x_5^{2^{d-1}-2} $\cr 
365. & $x_1^{7}x_2^{2^d-5}x_3^{2^d-1}x_4^{2^{d-1}-3}x_5^{2^{d+1}-2^{d-1}-2} $& 366. & $x_1^{7}x_2^{2^d-5}x_3^{2^d-1}x_4^{2^{d+1}-2^{d-1}-3}x_5^{2^{d-1}-2} $\cr 
367. & $x_1^{7}x_2^{2^d-5}x_3^{2^{d+1}-2^{d-1}-3}x_4^{2^{d-1}-2}x_5^{2^d-1} $& 368. & $x_1^{7}x_2^{2^d-5}x_3^{2^{d+1}-2^{d-1}-3}x_4^{2^{d-1}-1}x_5^{2^d-2} $\cr 
 \end{tabular}}
 \centerline{\begin{tabular}{llll}
369. & $x_1^{7}x_2^{2^d-5}x_3^{2^{d+1}-2^{d-1}-3}x_4^{2^d-2}x_5^{2^{d-1}-1} $& 370. & $x_1^{7}x_2^{2^d-5}x_3^{2^{d+1}-2^{d-1}-3}x_4^{2^d-1}x_5^{2^{d-1}-2} $\cr 
371. & $x_1^{7}x_2^{2^d-5}x_3^{2^{d+1}-2^{d-1}-1}x_4^{2^{d-1}-3}x_5^{2^d-2} $& 372. & $x_1^{7}x_2^{2^d-5}x_3^{2^{d+1}-2^{d-1}-1}x_4^{2^d-3}x_5^{2^{d-1}-2} $\cr 
373. & $x_1^{7}x_2^{2^d-5}x_3^{2^{d+1}-3}x_4^{2^{d-1}-2}x_5^{2^{d-1}-1} $& 374. & $x_1^{7}x_2^{2^d-5}x_3^{2^{d+1}-3}x_4^{2^{d-1}-1}x_5^{2^{d-1}-2} $\cr 
375. & $x_1^{7}x_2^{2^d-5}x_3^{2^{d+1}-1}x_4^{2^{d-1}-3}x_5^{2^{d-1}-2} $& 376. & $x_1^{7}x_2^{2^d-1}x_3^{2^d-5}x_4^{2^{d-1}-3}x_5^{2^{d+1}-2^{d-1}-2} $\cr 
377. & $x_1^{7}x_2^{2^d-1}x_3^{2^d-5}x_4^{2^{d+1}-2^{d-1}-3}x_5^{2^{d-1}-2} $& 378. & $x_1^{7}x_2^{2^d-1}x_3^{2^{d+1}-2^{d-1}-5}x_4^{2^{d-1}-3}x_5^{2^d-2} $\cr 
379. & $x_1^{7}x_2^{2^d-1}x_3^{2^{d+1}-2^{d-1}-5}x_4^{2^d-3}x_5^{2^{d-1}-2} $& 380. & $x_1^{7}x_2^{2^d-1}x_3^{2^{d+1}-5}x_4^{2^{d-1}-3}x_5^{2^{d-1}-2} $\cr 
381. & $x_1^{7}x_2^{2^{d+1}-5}x_3^{2^{d-1}-3}x_4^{2^{d-1}-2}x_5^{2^d-1} $& 382. & $x_1^{7}x_2^{2^{d+1}-5}x_3^{2^{d-1}-3}x_4^{2^{d-1}-1}x_5^{2^d-2} $\cr 
383. & $x_1^{7}x_2^{2^{d+1}-5}x_3^{2^{d-1}-3}x_4^{2^d-2}x_5^{2^{d-1}-1} $& 384. & $x_1^{7}x_2^{2^{d+1}-5}x_3^{2^{d-1}-3}x_4^{2^d-1}x_5^{2^{d-1}-2} $\cr 
385. & $x_1^{7}x_2^{2^{d+1}-5}x_3^{2^{d-1}-1}x_4^{2^{d-1}-3}x_5^{2^d-2} $& 386. & $x_1^{7}x_2^{2^{d+1}-5}x_3^{2^{d-1}-1}x_4^{2^d-3}x_5^{2^{d-1}-2} $\cr 
387. & $x_1^{7}x_2^{2^{d+1}-5}x_3^{2^d-1}x_4^{2^{d-1}-3}x_5^{2^{d-1}-2} $& 388. & $x_1^{7}x_2^{2^{d+1}-1}x_3^{2^d-5}x_4^{2^{d-1}-3}x_5^{2^{d-1}-2} $\cr 
389. & $x_1^{15}x_2^{2^{d+1}-9}x_3^{2^d-5}x_4^{2^{d-1}-3}x_5^{2^{d-1}-2} $& 390. & $x_1^{2^{d-1}-1}x_2^{3}x_3^{2^{d-1}-3}x_4^{2^d-2}x_5^{2^{d+1}-1} $\cr 
391. & $x_1^{2^{d-1}-1}x_2^{3}x_3^{2^{d-1}-3}x_4^{2^d-1}x_5^{2^{d+1}-2} $& 392. & $x_1^{2^{d-1}-1}x_2^{3}x_3^{2^{d-1}-3}x_4^{2^{d+1}-2}x_5^{2^d-1} $\cr 
393. & $x_1^{2^{d-1}-1}x_2^{3}x_3^{2^{d-1}-3}x_4^{2^{d+1}-1}x_5^{2^d-2} $& 394. & $x_1^{2^{d-1}-1}x_2^{3}x_3^{2^{d-1}-1}x_4^{2^d-3}x_5^{2^{d+1}-2} $\cr 
395. & $x_1^{2^{d-1}-1}x_2^{3}x_3^{2^{d-1}-1}x_4^{2^{d+1}-3}x_5^{2^d-2} $& 396. & $x_1^{2^{d-1}-1}x_2^{3}x_3^{2^d-3}x_4^{2^{d-1}-2}x_5^{2^{d+1}-1} $\cr 
397. & $x_1^{2^{d-1}-1}x_2^{3}x_3^{2^d-3}x_4^{2^{d-1}-1}x_5^{2^{d+1}-2} $& 398. & $x_1^{2^{d-1}-1}x_2^{3}x_3^{2^d-3}x_4^{2^d-2}x_5^{2^{d+1}-2^{d-1}-1} $\cr 
399. & $x_1^{2^{d-1}-1}x_2^{3}x_3^{2^d-3}x_4^{2^d-1}x_5^{2^{d+1}-2^{d-1}-2} $& 400. & $x_1^{2^{d-1}-1}x_2^{3}x_3^{2^d-3}x_4^{2^{d+1}-2^{d-1}-2}x_5^{2^d-1} $\cr 
401. & $x_1^{2^{d-1}-1}x_2^{3}x_3^{2^d-3}x_4^{2^{d+1}-2^{d-1}-1}x_5^{2^d-2} $& 402. & $x_1^{2^{d-1}-1}x_2^{3}x_3^{2^d-3}x_4^{2^{d+1}-2}x_5^{2^{d-1}-1} $\cr 
403. & $x_1^{2^{d-1}-1}x_2^{3}x_3^{2^d-3}x_4^{2^{d+1}-1}x_5^{2^{d-1}-2} $& 404. & $x_1^{2^{d-1}-1}x_2^{3}x_3^{2^d-1}x_4^{2^{d-1}-3}x_5^{2^{d+1}-2} $\cr 
405. & $x_1^{2^{d-1}-1}x_2^{3}x_3^{2^d-1}x_4^{2^d-3}x_5^{2^{d+1}-2^{d-1}-2} $& 406. & $x_1^{2^{d-1}-1}x_2^{3}x_3^{2^d-1}x_4^{2^{d+1}-2^{d-1}-3}x_5^{2^d-2} $\cr 
407. & $x_1^{2^{d-1}-1}x_2^{3}x_3^{2^d-1}x_4^{2^{d+1}-3}x_5^{2^{d-1}-2} $& 408. & $x_1^{2^{d-1}-1}x_2^{3}x_3^{2^{d+1}-3}x_4^{2^{d-1}-2}x_5^{2^d-1} $\cr 
409. & $x_1^{2^{d-1}-1}x_2^{3}x_3^{2^{d+1}-3}x_4^{2^{d-1}-1}x_5^{2^d-2} $& 410. & $x_1^{2^{d-1}-1}x_2^{3}x_3^{2^{d+1}-3}x_4^{2^d-2}x_5^{2^{d-1}-1} $\cr 
411. & $x_1^{2^{d-1}-1}x_2^{3}x_3^{2^{d+1}-3}x_4^{2^d-1}x_5^{2^{d-1}-2} $& 412. & $x_1^{2^{d-1}-1}x_2^{3}x_3^{2^{d+1}-1}x_4^{2^{d-1}-3}x_5^{2^d-2} $\cr 
413. & $x_1^{2^{d-1}-1}x_2^{3}x_3^{2^{d+1}-1}x_4^{2^d-3}x_5^{2^{d-1}-2} $& 414. & $x_1^{2^{d-1}-1}x_2^{2^{d-1}-1}x_3^{3}x_4^{2^d-3}x_5^{2^{d+1}-2} $\cr 
415. & $x_1^{2^{d-1}-1}x_2^{2^{d-1}-1}x_3^{3}x_4^{2^{d+1}-3}x_5^{2^d-2} $& 416. & $x_1^{2^{d-1}-1}x_2^{2^d-1}x_3^{3}x_4^{2^{d-1}-3}x_5^{2^{d+1}-2} $\cr 
417. & $x_1^{2^{d-1}-1}x_2^{2^d-1}x_3^{3}x_4^{2^d-3}x_5^{2^{d+1}-2^{d-1}-2} $& 418. & $x_1^{2^{d-1}-1}x_2^{2^d-1}x_3^{3}x_4^{2^{d+1}-2^{d-1}-3}x_5^{2^d-2} $\cr 
419. & $x_1^{2^{d-1}-1}x_2^{2^d-1}x_3^{3}x_4^{2^{d+1}-3}x_5^{2^{d-1}-2} $& 420. & $x_1^{2^{d-1}-1}x_2^{2^{d+1}-1}x_3^{3}x_4^{2^{d-1}-3}x_5^{2^d-2} $\cr 
421. & $x_1^{2^{d-1}-1}x_2^{2^{d+1}-1}x_3^{3}x_4^{2^d-3}x_5^{2^{d-1}-2} $& 422. & $x_1^{2^d-1}x_2^{3}x_3^{2^{d-1}-3}x_4^{2^{d-1}-2}x_5^{2^{d+1}-1} $\cr 
423. & $x_1^{2^d-1}x_2^{3}x_3^{2^{d-1}-3}x_4^{2^{d-1}-1}x_5^{2^{d+1}-2} $& 424. & $x_1^{2^d-1}x_2^{3}x_3^{2^{d-1}-3}x_4^{2^d-2}x_5^{2^{d+1}-2^{d-1}-1} $\cr 
425. & $x_1^{2^d-1}x_2^{3}x_3^{2^{d-1}-3}x_4^{2^d-1}x_5^{2^{d+1}-2^{d-1}-2} $& 426. & $x_1^{2^d-1}x_2^{3}x_3^{2^{d-1}-3}x_4^{2^{d+1}-2^{d-1}-2}x_5^{2^d-1} $\cr 
427. & $x_1^{2^d-1}x_2^{3}x_3^{2^{d-1}-3}x_4^{2^{d+1}-2^{d-1}-1}x_5^{2^d-2} $& 428. & $x_1^{2^d-1}x_2^{3}x_3^{2^{d-1}-3}x_4^{2^{d+1}-2}x_5^{2^{d-1}-1} $\cr 
429. & $x_1^{2^d-1}x_2^{3}x_3^{2^{d-1}-3}x_4^{2^{d+1}-1}x_5^{2^{d-1}-2} $& 430. & $x_1^{2^d-1}x_2^{3}x_3^{2^{d-1}-1}x_4^{2^{d-1}-3}x_5^{2^{d+1}-2} $\cr 
431. & $x_1^{2^d-1}x_2^{3}x_3^{2^d-1}x_4^{2^{d-1}-3}x_5^{2^{d+1}-2^{d-1}-2} $& 432. & $x_1^{2^d-1}x_2^{3}x_3^{2^{d+1}-2^{d-1}-1}x_4^{2^{d-1}-3}x_5^{2^d-2} $\cr 
433. & $x_1^{2^d-1}x_2^{3}x_3^{2^{d+1}-1}x_4^{2^{d-1}-3}x_5^{2^{d-1}-2} $& 434. & $x_1^{2^d-1}x_2^{7}x_3^{2^d-5}x_4^{2^{d-1}-3}x_5^{2^{d+1}-2^{d-1}-2} $\cr 
435. & $x_1^{2^d-1}x_2^{7}x_3^{2^d-5}x_4^{2^{d+1}-2^{d-1}-3}x_5^{2^{d-1}-2} $& 436. & $x_1^{2^d-1}x_2^{7}x_3^{2^{d+1}-2^{d-1}-5}x_4^{2^{d-1}-3}x_5^{2^d-2} $\cr 
437. & $x_1^{2^d-1}x_2^{7}x_3^{2^{d+1}-2^{d-1}-5}x_4^{2^d-3}x_5^{2^{d-1}-2} $& 438. & $x_1^{2^d-1}x_2^{7}x_3^{2^{d+1}-5}x_4^{2^{d-1}-3}x_5^{2^{d-1}-2} $\cr 
439. & $x_1^{2^d-1}x_2^{2^{d-1}-1}x_3^{3}x_4^{2^{d-1}-3}x_5^{2^{d+1}-2} $& 440. & $x_1^{2^d-1}x_2^{2^{d-1}-1}x_3^{3}x_4^{2^d-3}x_5^{2^{d+1}-2^{d-1}-2} $\cr 
441. & $x_1^{2^d-1}x_2^{2^{d-1}-1}x_3^{3}x_4^{2^{d+1}-2^{d-1}-3}x_5^{2^d-2} $& 442. & $x_1^{2^d-1}x_2^{2^{d-1}-1}x_3^{3}x_4^{2^{d+1}-3}x_5^{2^{d-1}-2} $\cr 
443. & $x_1^{2^d-1}x_2^{2^d-1}x_3^{3}x_4^{2^{d-1}-3}x_5^{2^{d+1}-2^{d-1}-2} $& 444. & $x_1^{2^d-1}x_2^{2^{d+1}-2^{d-1}-1}x_3^{3}x_4^{2^{d-1}-3}x_5^{2^d-2} $\cr 
445. & $x_1^{2^d-1}x_2^{2^{d+1}-1}x_3^{3}x_4^{2^{d-1}-3}x_5^{2^{d-1}-2} $& 446. & $x_1^{2^{d+1}-1}x_2^{3}x_3^{2^{d-1}-3}x_4^{2^{d-1}-2}x_5^{2^d-1} $\cr 
447. & $x_1^{2^{d+1}-1}x_2^{3}x_3^{2^{d-1}-3}x_4^{2^{d-1}-1}x_5^{2^d-2} $& 448. & $x_1^{2^{d+1}-1}x_2^{3}x_3^{2^{d-1}-3}x_4^{2^d-2}x_5^{2^{d-1}-1} $\cr 
 \end{tabular}}
 \centerline{\begin{tabular}{llll}
449. & $x_1^{2^{d+1}-1}x_2^{3}x_3^{2^{d-1}-3}x_4^{2^d-1}x_5^{2^{d-1}-2} $& 450. & $x_1^{2^{d+1}-1}x_2^{3}x_3^{2^{d-1}-1}x_4^{2^{d-1}-3}x_5^{2^d-2} $\cr 
451. & $x_1^{2^{d+1}-1}x_2^{3}x_3^{2^d-1}x_4^{2^{d-1}-3}x_5^{2^{d-1}-2} $& 452. & $x_1^{2^{d+1}-1}x_2^{7}x_3^{2^d-5}x_4^{2^{d-1}-3}x_5^{2^{d-1}-2} $\cr 
453. & $x_1^{2^{d+1}-1}x_2^{2^{d-1}-1}x_3^{3}x_4^{2^{d-1}-3}x_5^{2^d-2} $& 454. & $x_1^{2^{d+1}-1}x_2^{2^{d-1}-1}x_3^{3}x_4^{2^d-3}x_5^{2^{d-1}-2} $\cr 
455. & $x_1^{2^{d+1}-1}x_2^{2^d-1}x_3^{3}x_4^{2^{d-1}-3}x_5^{2^{d-1}-2} $&
 \end{tabular}}

\medskip
For $d = 4$,

\medskip
 \centerline{\begin{tabular}{llll}
456. \ $x_1^{7}x_2^{7}x_3^{7}x_4^{8}x_5^{31} $& 457. \ $x_1^{7}x_2^{7}x_3^{7}x_4^{9}x_5^{30} $& 458. \ $x_1^{7}x_2^{7}x_3^{7}x_4^{15}x_5^{24} $& 459. \ $x_1^{7}x_2^{7}x_3^{7}x_4^{24}x_5^{15} $\cr  
460. \ $x_1^{7}x_2^{7}x_3^{7}x_4^{25}x_5^{14} $& 461. \ $x_1^{7}x_2^{7}x_3^{7}x_4^{31}x_5^{8} $& 462. \ $x_1^{7}x_2^{7}x_3^{9}x_4^{6}x_5^{31} $& 463. \ $x_1^{7}x_2^{7}x_3^{9}x_4^{7}x_5^{30} $\cr  
464. \ $x_1^{7}x_2^{7}x_3^{9}x_4^{14}x_5^{23} $& 465. \ $x_1^{7}x_2^{7}x_3^{9}x_4^{15}x_5^{22} $& 466. \ $x_1^{7}x_2^{7}x_3^{9}x_4^{22}x_5^{15} $& 467. \ $x_1^{7}x_2^{7}x_3^{9}x_4^{23}x_5^{14} $\cr  
468. \ $x_1^{7}x_2^{7}x_3^{9}x_4^{30}x_5^{7} $& 469. \ $x_1^{7}x_2^{7}x_3^{9}x_4^{31}x_5^{6} $& 470. \ $x_1^{7}x_2^{7}x_3^{15}x_4^{7}x_5^{24} $& 471. \ $x_1^{7}x_2^{7}x_3^{15}x_4^{9}x_5^{22} $\cr  
472. \ $x_1^{7}x_2^{7}x_3^{15}x_4^{15}x_5^{16} $& 473. \ $x_1^{7}x_2^{7}x_3^{15}x_4^{17}x_5^{14} $& 474. \ $x_1^{7}x_2^{7}x_3^{15}x_4^{23}x_5^{8} $& 475. \ $x_1^{7}x_2^{7}x_3^{15}x_4^{25}x_5^{6} $\cr  
476. \ $x_1^{7}x_2^{7}x_3^{25}x_4^{6}x_5^{15} $& 477. \ $x_1^{7}x_2^{7}x_3^{25}x_4^{7}x_5^{14} $& 478. \ $x_1^{7}x_2^{7}x_3^{25}x_4^{14}x_5^{7} $& 479. \ $x_1^{7}x_2^{7}x_3^{25}x_4^{15}x_5^{6} $\cr  
480. \ $x_1^{7}x_2^{7}x_3^{31}x_4^{7}x_5^{8} $& 481. \ $x_1^{7}x_2^{7}x_3^{31}x_4^{9}x_5^{6} $& 482. \ $x_1^{7}x_2^{15}x_3^{7}x_4^{7}x_5^{24} $& 483. \ $x_1^{7}x_2^{15}x_3^{7}x_4^{9}x_5^{22} $\cr  
484. \ $x_1^{7}x_2^{15}x_3^{7}x_4^{15}x_5^{16} $& 485. \ $x_1^{7}x_2^{15}x_3^{7}x_4^{17}x_5^{14} $& 486. \ $x_1^{7}x_2^{15}x_3^{7}x_4^{23}x_5^{8} $& 487. \ $x_1^{7}x_2^{15}x_3^{7}x_4^{25}x_5^{6} $\cr  
488. \ $x_1^{7}x_2^{15}x_3^{15}x_4^{7}x_5^{16} $& 489. \ $x_1^{7}x_2^{15}x_3^{15}x_4^{16}x_5^{7} $& 490. \ $x_1^{7}x_2^{15}x_3^{15}x_4^{17}x_5^{6} $& 491. \ $x_1^{7}x_2^{15}x_3^{17}x_4^{7}x_5^{14} $\cr  
492. \ $x_1^{7}x_2^{15}x_3^{23}x_4^{7}x_5^{8} $& 493. \ $x_1^{7}x_2^{15}x_3^{23}x_4^{9}x_5^{6} $& 494. \ $x_1^{7}x_2^{31}x_3^{7}x_4^{7}x_5^{8} $& 495. \ $x_1^{7}x_2^{31}x_3^{7}x_4^{9}x_5^{6} $\cr  
496. \ $x_1^{15}x_2^{7}x_3^{7}x_4^{7}x_5^{24} $& 497. \ $x_1^{15}x_2^{7}x_3^{7}x_4^{9}x_5^{22} $& 498. \ $x_1^{15}x_2^{7}x_3^{7}x_4^{15}x_5^{16} $& 499. \ $x_1^{15}x_2^{7}x_3^{7}x_4^{17}x_5^{14} $\cr  
500. \ $x_1^{15}x_2^{7}x_3^{7}x_4^{23}x_5^{8} $& 501. \ $x_1^{15}x_2^{7}x_3^{7}x_4^{25}x_5^{6} $& 502. \ $x_1^{15}x_2^{7}x_3^{15}x_4^{7}x_5^{16} $& 503. \ $x_1^{15}x_2^{7}x_3^{15}x_4^{16}x_5^{7} $\cr  
504. \ $x_1^{15}x_2^{7}x_3^{15}x_4^{17}x_5^{6} $& 505. \ $x_1^{15}x_2^{7}x_3^{17}x_4^{7}x_5^{14} $& 506. \ $x_1^{15}x_2^{7}x_3^{23}x_4^{7}x_5^{8} $& 507. \ $x_1^{15}x_2^{7}x_3^{23}x_4^{9}x_5^{6} $\cr  
508. \ $x_1^{15}x_2^{15}x_3^{7}x_4^{7}x_5^{16} $& 509. \ $x_1^{15}x_2^{15}x_3^{7}x_4^{16}x_5^{7} $& 510. \ $x_1^{15}x_2^{15}x_3^{7}x_4^{17}x_5^{6} $& 511. \ $x_1^{15}x_2^{15}x_3^{17}x_4^{6}x_5^{7} $\cr  
512. \ $x_1^{15}x_2^{15}x_3^{17}x_4^{7}x_5^{6} $& 513. \ $x_1^{15}x_2^{15}x_3^{19}x_4^{5}x_5^{6} $& 514. \ $x_1^{15}x_2^{19}x_3^{5}x_4^{7}x_5^{14} $& 515. \ $x_1^{15}x_2^{19}x_3^{7}x_4^{5}x_5^{14} $\cr  
516. \ $x_1^{15}x_2^{19}x_3^{7}x_4^{13}x_5^{6} $& 517. \ $x_1^{15}x_2^{23}x_3^{7}x_4^{7}x_5^{8} $& 518. \ $x_1^{15}x_2^{23}x_3^{7}x_4^{9}x_5^{6} $& 519. \ $x_1^{31}x_2^{7}x_3^{7}x_4^{7}x_5^{8} $\cr  
520. \ $x_1^{31}x_2^{7}x_3^{7}x_4^{9}x_5^{6} $& 
 \end{tabular}}

\medskip
For $d \geqslant 5$,

\medskip
 \centerline{\begin{tabular}{llll}
456. & $x_1^{7}x_2^{2^{d-1}-5}x_3^{2^{d-1}-3}x_4^{2^d-2}x_5^{2^{d+1}-1} $& 457. & $x_1^{7}x_2^{2^{d-1}-5}x_3^{2^{d-1}-3}x_4^{2^d-1}x_5^{2^{d+1}-2} $\cr 
458. & $x_1^{7}x_2^{2^{d-1}-5}x_3^{2^{d-1}-3}x_4^{2^{d+1}-2}x_5^{2^d-1} $& 459. & $x_1^{7}x_2^{2^{d-1}-5}x_3^{2^{d-1}-3}x_4^{2^{d+1}-1}x_5^{2^d-2} $\cr 
460. & $x_1^{7}x_2^{2^{d-1}-5}x_3^{2^{d-1}-1}x_4^{2^d-3}x_5^{2^{d+1}-2} $& 461. & $x_1^{7}x_2^{2^{d-1}-5}x_3^{2^{d-1}-1}x_4^{2^{d+1}-3}x_5^{2^d-2} $\cr 
462. & $x_1^{7}x_2^{2^{d-1}-5}x_3^{2^d-3}x_4^{2^{d-1}-2}x_5^{2^{d+1}-1} $& 463. & $x_1^{7}x_2^{2^{d-1}-5}x_3^{2^d-3}x_4^{2^{d-1}-1}x_5^{2^{d+1}-2} $\cr 
464. & $x_1^{7}x_2^{2^{d-1}-5}x_3^{2^d-3}x_4^{2^d-2}x_5^{2^{d+1}-2^{d-1}-1} $& 465. & $x_1^{7}x_2^{2^{d-1}-5}x_3^{2^d-3}x_4^{2^d-1}x_5^{2^{d+1}-2^{d-1}-2} $\cr 
466. & $x_1^{7}x_2^{2^{d-1}-5}x_3^{2^d-3}x_4^{2^{d+1}-2^{d-1}-2}x_5^{2^d-1} $& 467. & $x_1^{7}x_2^{2^{d-1}-5}x_3^{2^d-3}x_4^{2^{d+1}-2^{d-1}-1}x_5^{2^d-2} $\cr 
468. & $x_1^{7}x_2^{2^{d-1}-5}x_3^{2^d-3}x_4^{2^{d+1}-2}x_5^{2^{d-1}-1} $& 469. & $x_1^{7}x_2^{2^{d-1}-5}x_3^{2^d-3}x_4^{2^{d+1}-1}x_5^{2^{d-1}-2} $\cr 
470. & $x_1^{7}x_2^{2^{d-1}-5}x_3^{2^d-1}x_4^{2^{d-1}-3}x_5^{2^{d+1}-2} $& 471. & $x_1^{7}x_2^{2^{d-1}-5}x_3^{2^d-1}x_4^{2^d-3}x_5^{2^{d+1}-2^{d-1}-2} $\cr 
472. & $x_1^{7}x_2^{2^{d-1}-5}x_3^{2^d-1}x_4^{2^{d+1}-2^{d-1}-3}x_5^{2^d-2} $& 473. & $x_1^{7}x_2^{2^{d-1}-5}x_3^{2^d-1}x_4^{2^{d+1}-3}x_5^{2^{d-1}-2} $\cr 
474. & $x_1^{7}x_2^{2^{d-1}-5}x_3^{2^{d+1}-3}x_4^{2^{d-1}-2}x_5^{2^d-1} $& 475. & $x_1^{7}x_2^{2^{d-1}-5}x_3^{2^{d+1}-3}x_4^{2^{d-1}-1}x_5^{2^d-2} $\cr 
476. & $x_1^{7}x_2^{2^{d-1}-5}x_3^{2^{d+1}-3}x_4^{2^d-2}x_5^{2^{d-1}-1} $& 477. & $x_1^{7}x_2^{2^{d-1}-5}x_3^{2^{d+1}-3}x_4^{2^d-1}x_5^{2^{d-1}-2} $\cr 
478. & $x_1^{7}x_2^{2^{d-1}-5}x_3^{2^{d+1}-1}x_4^{2^{d-1}-3}x_5^{2^d-2} $& 479. & $x_1^{7}x_2^{2^{d-1}-5}x_3^{2^{d+1}-1}x_4^{2^d-3}x_5^{2^{d-1}-2} $\cr 
480. & $x_1^{7}x_2^{2^{d-1}-1}x_3^{2^{d-1}-5}x_4^{2^d-3}x_5^{2^{d+1}-2} $& 481. & $x_1^{7}x_2^{2^{d-1}-1}x_3^{2^{d-1}-5}x_4^{2^{d+1}-3}x_5^{2^d-2} $\cr 
482. & $x_1^{7}x_2^{2^d-1}x_3^{2^{d-1}-5}x_4^{2^{d-1}-3}x_5^{2^{d+1}-2} $& 483. & $x_1^{7}x_2^{2^d-1}x_3^{2^{d-1}-5}x_4^{2^d-3}x_5^{2^{d+1}-2^{d-1}-2} $\cr 
484. & $x_1^{7}x_2^{2^d-1}x_3^{2^{d-1}-5}x_4^{2^{d+1}-2^{d-1}-3}x_5^{2^d-2} $& 485. & $x_1^{7}x_2^{2^d-1}x_3^{2^{d-1}-5}x_4^{2^{d+1}-3}x_5^{2^{d-1}-2} $\cr 
486. & $x_1^{7}x_2^{2^{d+1}-1}x_3^{2^{d-1}-5}x_4^{2^{d-1}-3}x_5^{2^d-2} $& 487. & $x_1^{7}x_2^{2^{d+1}-1}x_3^{2^{d-1}-5}x_4^{2^d-3}x_5^{2^{d-1}-2} $\cr 
488. & $x_1^{15}x_2^{2^d-9}x_3^{2^{d-1}-5}x_4^{2^{d-1}-3}x_5^{2^{d+1}-2} $& 489. & $x_1^{15}x_2^{2^d-9}x_3^{2^{d-1}-5}x_4^{2^d-3}x_5^{2^{d+1}-2^{d-1}-2} $\cr 
490. & $x_1^{15}x_2^{2^d-9}x_3^{2^{d-1}-5}x_4^{2^{d+1}-2^{d-1}-3}x_5^{2^d-2} $& 491. & $x_1^{15}x_2^{2^d-9}x_3^{2^{d-1}-5}x_4^{2^{d+1}-3}x_5^{2^{d-1}-2} $\cr 
492. & $x_1^{15}x_2^{2^d-9}x_3^{2^d-5}x_4^{2^{d-1}-3}x_5^{2^{d+1}-2^{d-1}-2} $& 493. & $x_1^{15}x_2^{2^d-9}x_3^{2^d-5}x_4^{2^{d+1}-2^{d-1}-3}x_5^{2^{d-1}-2} $\cr 
 \end{tabular}}
 \centerline{\begin{tabular}{llll}
494. & $x_1^{15}x_2^{2^d-9}x_3^{2^{d+1}-2^{d-1}-5}x_4^{2^{d-1}-3}x_5^{2^d-2} $& 495. & $x_1^{15}x_2^{2^d-9}x_3^{2^{d+1}-2^{d-1}-5}x_4^{2^d-3}x_5^{2^{d-1}-2} $\cr 
496. & $x_1^{15}x_2^{2^d-9}x_3^{2^{d+1}-5}x_4^{2^{d-1}-3}x_5^{2^{d-1}-2} $& 497. & $x_1^{15}x_2^{2^{d+1}-9}x_3^{2^{d-1}-5}x_4^{2^{d-1}-3}x_5^{2^d-2} $\cr 
498. & $x_1^{15}x_2^{2^{d+1}-9}x_3^{2^{d-1}-5}x_4^{2^d-3}x_5^{2^{d-1}-2} $& 499. & $x_1^{2^{d-1}-1}x_2^{7}x_3^{2^{d-1}-5}x_4^{2^d-3}x_5^{2^{d+1}-2} $\cr 
500. & $x_1^{2^{d-1}-1}x_2^{7}x_3^{2^{d-1}-5}x_4^{2^{d+1}-3}x_5^{2^d-2} $& 501. & $x_1^{2^{d-1}-1}x_2^{7}x_3^{2^d-5}x_4^{2^{d-1}-3}x_5^{2^{d+1}-2} $\cr 
502. & $x_1^{2^{d-1}-1}x_2^{7}x_3^{2^d-5}x_4^{2^d-3}x_5^{2^{d+1}-2^{d-1}-2} $& 503. & $x_1^{2^{d-1}-1}x_2^{7}x_3^{2^d-5}x_4^{2^{d+1}-2^{d-1}-3}x_5^{2^d-2} $\cr 
504. & $x_1^{2^{d-1}-1}x_2^{7}x_3^{2^d-5}x_4^{2^{d+1}-3}x_5^{2^{d-1}-2} $& 505. & $x_1^{2^{d-1}-1}x_2^{7}x_3^{2^{d+1}-5}x_4^{2^{d-1}-3}x_5^{2^d-2} $\cr 
506. & $x_1^{2^{d-1}-1}x_2^{7}x_3^{2^{d+1}-5}x_4^{2^d-3}x_5^{2^{d-1}-2} $& 507. & $x_1^{2^d-1}x_2^{7}x_3^{2^{d-1}-5}x_4^{2^{d-1}-3}x_5^{2^{d+1}-2} $\cr 
508. & $x_1^{2^d-1}x_2^{7}x_3^{2^{d-1}-5}x_4^{2^d-3}x_5^{2^{d+1}-2^{d-1}-2} $& 509. & $x_1^{2^d-1}x_2^{7}x_3^{2^{d-1}-5}x_4^{2^{d+1}-2^{d-1}-3}x_5^{2^d-2} $\cr 
510. & $x_1^{2^d-1}x_2^{7}x_3^{2^{d-1}-5}x_4^{2^{d+1}-3}x_5^{2^{d-1}-2} $& 511. & $x_1^{2^{d+1}-1}x_2^{7}x_3^{2^{d-1}-5}x_4^{2^{d-1}-3}x_5^{2^d-2} $\cr 
512. & $x_1^{2^{d+1}-1}x_2^{7}x_3^{2^{d-1}-5}x_4^{2^d-3}x_5^{2^{d-1}-2} $& 
 \end{tabular}}

\medskip
For $d = 5$,

\medskip
 \centerline{\begin{tabular}{lll}
513. \ $x_1^{15}x_2^{15}x_3^{15}x_4^{17}x_5^{62} $& 514. \ $x_1^{15}x_2^{15}x_3^{15}x_4^{49}x_5^{30} $& 515. \ $x_1^{15}x_2^{15}x_3^{19}x_4^{13}x_5^{62} $\cr 
516. \ $x_1^{15}x_2^{15}x_3^{19}x_4^{29}x_5^{46} $&  517. \ $x_1^{15}x_2^{15}x_3^{19}x_4^{45}x_5^{30} $& 518. \ $x_1^{15}x_2^{15}x_3^{19}x_4^{61}x_5^{14} $\cr  
519. \ $x_1^{15}x_2^{15}x_3^{51}x_4^{13}x_5^{30} $& 520. \ $x_1^{15}x_2^{15}x_3^{51}x_4^{29}x_5^{14} $&\cr 
 \end{tabular}}

\medskip
For $d \geqslant 6$,

\medskip
 \centerline{\begin{tabular}{llll}
513. & $x_1^{15}x_2^{2^{d-1}-9}x_3^{2^{d-1}-5}x_4^{2^d-3}x_5^{2^{d+1}-2} $& 514. & $x_1^{15}x_2^{2^{d-1}-9}x_3^{2^{d-1}-5}x_4^{2^{d+1}-3}x_5^{2^d-2} $\cr 515. & $x_1^{15}x_2^{2^{d-1}-9}x_3^{2^d-5}x_4^{2^{d-1}-3}x_5^{2^{d+1}-2} $& 516. & $x_1^{15}x_2^{2^{d-1}-9}x_3^{2^d-5}x_4^{2^d-3}x_5^{2^{d+1}-2^{d-1}-2} $\cr 517. & $x_1^{15}x_2^{2^{d-1}-9}x_3^{2^d-5}x_4^{2^{d+1}-2^{d-1}-3}x_5^{2^d-2} $& 518. & $x_1^{15}x_2^{2^{d-1}-9}x_3^{2^d-5}x_4^{2^{d+1}-3}x_5^{2^{d-1}-2} $\cr 519. & $x_1^{15}x_2^{2^{d-1}-9}x_3^{2^{d+1}-5}x_4^{2^{d-1}-3}x_5^{2^d-2} $& 520. & $x_1^{15}x_2^{2^{d-1}-9}x_3^{2^{d+1}-5}x_4^{2^d-3}x_5^{2^{d-1}-2} $\cr 
 \end{tabular}}

\bigskip\noindent
{\bf Acknowledgment.} 
The final version of this paper was completed while the second named author was visiting the Vietnam Institute for Advanced Study in Mathematics (VIASM) from August to December, 2015. He would like to thank the VIASM for financial support and kind hospitality.

 We would like to express our warmest thanks to the referee for the careful reading and helpful suggestions.
{}

\medskip\noindent
\ \ {Department of Mathematics, Quy Nh\ohorn n University, 

\noindent
\ \ 170 An D\uhorn \ohorn ng V\uhorn \ohorn ng, Quy Nh\ohorn n, B\`inh \DJ \d inh, Vi\d\ecircumflex t Nam.}

\medskip
\noindent \ \ E-mail:  dangphuc150497@gmail.com and nguyensum@qnu.edu.vn

\end{document}